\documentclass[preprint]{imsart}

\RequirePackage[OT1]{fontenc}
\RequirePackage{amsthm,amsmath,graphicx}
\RequirePackage[sort,numbers]{natbib}
\RequirePackage[colorlinks,citecolor=red,urlcolor=blue]{hyperref}

\usepackage{amsfonts,amssymb}
\usepackage[ruled,vlined]{algorithm2e}
\usepackage[usenames,dvipsnames]{xcolor}

\usepackage{mathtools}

\startlocaldefs

\theoremstyle{plain}
\newtheorem{lemma}{Lemma}[section]

\newtheorem{theorem}{Theorem}[section]
\newtheorem{definition}{Definition}[section]

\newtheorem{conjecture}{Conjecture}[section]

\newcommand{\E} {\mathbb{E}}

\newcommand{\RR} {\mathbb{R}}
\newcommand{\PP} {\mathbb{P}}
\newcommand{\QQ} {\mathbb{Q}}

\newcommand{\tr}{\text{tr}}

\DeclareMathOperator*{\argmin}{arg\,min}
\DeclareMathOperator*{\argmax}{arg\,max}

\DeclareMathOperator*{\Ex}{\mathbb{E}}
\DeclareMathOperator*{\Var}{\mathrm{Var}}

\let\oldenumerate\enumerate
\renewcommand{\enumerate}{
  \oldenumerate
  \setlength{\itemsep}{1pt}
  \setlength{\parskip}{0pt}
  \setlength{\parsep}{0pt}
}

\newcommand{\mathr}{\mathbb{R}}

\newcommand{\Cov}[1]{\text{Cov}}

\definecolor{AfonsoBlue}{RGB}{30,65,123}

\endlocaldefs

\begin{document}
\begin{frontmatter}
\title{Computationally efficient sparse clustering}
\runtitle{Sparse clustering}

\begin{aug}
\author{\fnms{Matthias} \snm{L\"offler}\thanksref{a},
	\ead[label=e0]{matthias.loeffler@stat.math.ethz.ch}
}
\author{\fnms{Alexander S.} \snm{Wein}\thanksref{b}\ead[label=e2]{awein@cims.nyu.edu}
}
\and
\author{\fnms{Afonso S.} \snm{Bandeira}\thanksref{c}\ead[label=e1]{bandeira@math.ethz.ch}}

\runauthor{L\"offler, Wein, Bandeira}
\thankstext{a}{Funded in part by ETH Foundations of Data Science (ETH-FDS).}
\thankstext{b}{Partially supported by NSF grant DMS-1712730 and by the Simons Collaboration on Algorithms and Geometry.}

\address[a]{Seminar for Statistics,
Department of Mathematics, 
	ETH Z\"urich \\
			\printead{e0}
}

\address[b]{Department of Mathematics, Courant Institute of Mathematical Sciences, NYU \\
\printead{e2}
}

\address[c]{Department of Mathematics,	ETH Z\"urich \\
\printead{e1}}
\end{aug}

\begin{abstract}
	We study statistical and computational limits of clustering when the means of the centres are sparse and their dimension is possibly much larger than the sample size. Our theoretical analysis focuses on the model $X_i=z_i \theta+\varepsilon_{i}, ~z_i \in \{-1,1\}, ~\varepsilon_i \thicksim \mathcal{N}(0, I)$, {which has two clusters with centres $\theta$ and $-\theta$}.
	
	We provide a finite sample analysis of a new sparse clustering algorithm based on sparse PCA and show that it achieves the minimax optimal misclustering rate in the regime $\|\theta\| \rightarrow \infty$.
	
	Our results require the sparsity to grow slower than the square root of the sample size. Using a recent framework for computational lower bounds---the \emph{low-degree likelihood ratio}---we give evidence that this condition is necessary for any polynomial-time clustering algorithm to succeed below the BBP threshold. This complements existing evidence based on reductions and statistical query lower bounds. Compared to these existing results, we cover a wider set of parameter regimes and give a more precise understanding of the runtime required and the misclustering error achievable. Our results imply that a large class of tests based on low-degree polynomials fail to solve even the \emph{weak testing} task.

\end{abstract}



\end{frontmatter}


\section{Introduction}
Clustering data points $(X_1, \dots, X_n)$ into homogeneous groups is a fundamental and important data processing step in statistics and machine learning. In recent years, clustering in the high-dimensional settings has seen an increasing influx of attention; see for example \cite{BouveyronBrunetSaumard14} for a recent review. 

When the dimensionality of the data points, $p$, is large compared to the number of samples, $n$, consistent clustering is, in general, information theoretically impossible \cite{Ndaoud19}. Consequently, traditional clustering algorithms such as Lloyd's algorithm \cite{Lloyd82}, spectral clustering \cite{VempalaWang04, von2007tutorial}, SDP relaxations of $k$-means \cite{PengWei07}, the EM algorithm \cite{DempsterLairdRubin77} and approximate message passing \cite{LesieurDeBaccoBanksKrzakalaMooreZdeborova16} are poised to fail in this regime. To circumvent this issue, additional parsimony assumptions have to be taken into account. The prevalent approach is to assume sparsity of the cluster centres. For instance, \cite{FriedmanMeulman04,WittenTibshirani10} propose to use weighted versions of $k$-means where only a small number of features are considered and chosen by $\ell^1$-penalization. Similarly, \cite{PanShen07} and \cite{WangZhu08} propose to use $\ell^1$-penalized versions of $k$-means which are computed by iterative update steps. These algorithms work well empirically, but a sound theoretical treatment of them is lacking so far. 

Only more recently has there been a series of papers which provide algorithms with some theoretical guarantees, considering a sparse Gaussian mixture model
\begin{align*}
X_i=\theta_{z_i}+\varepsilon_i,  ~\varepsilon_i \overset{i.i.d}{\thicksim} \mathcal{N}(0, I_p), ~z_i \in \{1, \dots, k \}, ~ \left  | \bigcup_{j=1}^k \text{supp}(\theta_j)\right |\leq s
\end{align*}
and extensions of it.

Above the BBP transition \cite{BaikBenArousPeche05}, i.e.\ when $p=o(n \min_{j \neq l} \|\theta_l-\theta_j\|^4)$, it is not necessary to use the sparsity assumption. In particular, a SDP relaxation of $k$-means \cite{GiraudVerzelen19} and Lloyd's algorithm \cite{LuZhou16, Ndaoud19} have been shown to achieve minimax optimal misclustering rates. 

In contrast, below the BBP transition it is necessary to take the sparsity assumption into account and use modified algorithms \cite{JinKeWang17,Ndaoud19}. For instance, \cite{CaiMaZhang19} propose a high-dimensional modification of the EM-algorithm which takes sparsity into account. They show that the cluster centres are estimated at the optimal rate and give upper and lower bounds for the accuracy of predicting the label of a new observation. The same loss function is also considered in \cite{AzizyanSinghWasserman13} and \cite{AzizyanSinghWasserman15}. 

Another possibility is to select the relevant features first and afterwards use a vanilla clustering algorithm such as Lloyd's algorithm or $k$-means. This approach has been analyzed in \cite{AzizyanSinghWasserman13} and \cite{JinWang16,JinKeWang17}. Particularly, Jin et al.~\cite{JinKeWang17} develop a precise theory for the simplified model
\begin{align*}
X_i = z_i \theta+\varepsilon_{i}, ~~~z_i \in \{-1,1\},  ~~~\theta_j \overset{i.i.d}{\thicksim} (1-\epsilon) \delta_0+\epsilon \delta_\kappa/2+\epsilon \delta_{-\kappa}/2, ~~~\kappa \in \mathbb{R}
\end{align*} and show in which asymptotic regimes of $(p,n,\epsilon, \kappa)$ consistent clustering is possible when considering the misclustering error  
\begin{align}
\ell (\hat z, z):=\min_{\pi \in \{-1,1\}} \frac{1}{n} \sum_{i=1}^n \mathbf{1} \left ( \pi \hat z_i \neq z_i\right ). 
\end{align}

\noindent Moreover, Jin et al. \cite{JinKeWang17} conjecture the existence of a computational barrier, meaning a region of parameters $(p,n,\epsilon, \kappa)$ where consistent clustering is possible only when using algorithms that are not computable in polynomial time.

The phenomena of such computational barriers has been recently discovered in other sparse problems too, such as sparse PCA \cite{BerthetRigollet13b,WangBerthetSamworth16,HopkinsKothariPotechinRaghavendraSchrammSteurer17,BrennanBresler19b,DingKuniskyWeinBandeira19}, sparse CCA \cite{GaoMaZhou17}, sub-matrix detection \cite{MaWu15}, biclustering \cite{CaiLiangRakhlin17, BalakrishnanMolarRinaldoSinghWasserman11} and 
robust sparse mean detection \cite{BrennanBresler19,BrennanBresler20}. Most relevant to the present article, \cite{FanLiuWangYang18,BrennanBresler19} give evidence that detection in the sparse Gaussian mixture model is impossible in polynomial time under certain conditions which include $\|\theta\|^4=o(p/n) $, $\|\theta\|^4=o(s^2/n)$, and $n = o(p)$. It seems plausible that as a consequence, clustering better than with a random guess should also be hard in this regime, but a formal proof of this intuition is missing in the literature so far. Giving some evidence in this direction, \cite{FanLiuWangYang18} present a reduction which shows that, assuming hardness of the detection problem, no polynomial-time classification rule can near-perfectly match the output of Fisher's linear discriminant (the statistically optimal classifier) when predicting the label of a new observation.

In this article we further investigate statistical and computational limits of clustering in the high-dimensional limit, extending and building on the results of \cite{JinKeWang17,FanLiuWangYang18,BrennanBresler19,BrennanBresler20}. For our theoretical results we focus on the symmetric two-cluster model
\begin{equation}\label{eq:model}
X_i = z_i \theta + \varepsilon_{i}, ~~~\varepsilon_{i} \overset{i.i.d}{\thicksim} \mathcal{N}(0, I_p), ~~~z_i \in \{-1,1\}, ~~~\|\theta\|_0 \leq s, ~~~i=1, \dots, n.
\end{equation}

\noindent In particular, we show that when $s^2/n = o( \|\theta\|^4/\log p)$, a simple polynomial-time algorithm based on sparse PCA, in the spirit of \cite{AzizyanSinghWasserman15,JinWang16,JinKeWang17}, achieves the sharp exponential minimax optimal misclustering rate as $\|\theta\| \rightarrow \infty$. Notably, this rate coincides with the low-dimensional minimax rate without assuming sparsity \cite{LuZhou16}. 

We give evidence based on the recently proposed \emph{low-degree likelihood ratio} \cite{HopkinsSteurer17,HopkinsKothariPotechinRaghavendraSchrammSteurer17,Hopkins18} that when $p/n \ge \|\theta\|^4(1+\Omega(1)) $ (i.e., below the BBP transition) and $s^2/n=\omega( \|\theta\|^4 \log pn)$, no polynomial-time algorithm can distinguish the model~\eqref{eq:model} from i.i.d.\ Gaussian samples, corroborating and extending the existing computational lower bounds \cite{FanLiuWangYang18,BrennanBresler19}. More concretely, our results imply that a large class of tests based on low-degree polynomials fail to solve the distinguishing task. We furthermore give a reduction showing that if this distinguishing problem is indeed hard, then no polynomial-time algorithm can cluster better than a random guess. This is similar in spirit to existing detection-to-recovery reductions (see e.g., Section~5.1 of~\cite{MaWu15} or Section~10 of \cite{BBH}). Hence, the sample size requirement for our algorithm appears to be almost computationally optimal.

One of our conceptual contributions (see Section~\ref{sec:rig-ld}) is to show how bounds on the low-degree likelihood ratio rigorously imply failure of polynomial-based tests for the problem of \emph{weak testing} (i.e., beating a random guess). Previous results of this nature have focused only on \emph{consistent testing} (where the error probability must tend to zero).

\subsection*{Notation}
{We use standard linear algebra notation.}
$\|v\|_p$ denotes the usual $\ell^p$-norm of a vector $v$ and {$\|v\| := \|v\|_2$ }denotes the Euclidean norm. If $A$ is a matrix, $\|A\|$ denotes the spectral norm, $\| A\|_F$ the Frobenius (Hilbert-Schmidt) norm, $\| A \|_1:=\sum_{i,j} |A_{ij}|$ the entrywise $\ell^1$-norm and $\|A\|_\infty := \max_{i,j} |A_{ij}|$ the entrywise $\ell_\infty$-norm. We use the notation $x \lesssim y$ to denote that there exists a universal constant $c$, not dependent on $p, n, s$ or $\theta$ such that $x \leq cy$. Moreover, for matrices $A, B$ we write $A \preceq B$ if $B-A$ is positive semi-definite. For a vector $a$ we denote by $a_S$ the restriction of $a$ to the set $S \subset \{1, \dots, p\}$, i.e.\ $a_S:=(a_i)_{i \in S}$ and similarly for a matrix $A$, $A_S:=(A_{ij})_{i \in S, j \in S}$. For a projection matrix $P$ we denote $\text{supp}(P):=\{ i: ~P_{ii} \neq 0 \}$; {equivalently, $\text{supp}(P)$ is the union of supports of all vectors in the image of $P$}.
By $\mathcal{F}^k$ we denote the $k$-Fantope, $\mathcal{F}^k:=\{ P: ~\text{tr}(P)=k, ~0 \preceq P \preceq I, ~P^T=P \}$. We denote the indicator function by $\mathbf{1}(\cdot)$ and the sign function by $\text{sgn}(x):=\mathbf{1}(x > 0)-\mathbf{1}(x \leq 0)$. The notation $\mathbb{E}_{(\theta,z)}$ (or $\mathbb{P}_{(\theta,z)}$) denotes expectation (or probability, respectively) over samples $\{X_i\}$ drawn from the model~\eqref{eq:model} with parameters $\theta$ and $z$; when $\theta = 0$, we write simply $\mathbb{E}_0$, as $z$ becomes immaterial.
Our asymptotic notation (e.g.\ $o(1), \omega(1), O(1), \Omega(1)$) pertains to the limit $n \to \infty$, where the other parameters $p,s,\|\theta\|$ may depend on $n$.

\section{Symmetric two-cluster setting}
To illustrate our ideas and with the goal of providing concise and clear proofs, we restrict our considerations to the symmetric two-cluster sparse Gaussian mixture model
\begin{align} \label{sec 2 def gmm} \notag
X_i=z_i \theta+\varepsilon_i, ~~~\varepsilon_i \overset{i.i.d}{\thicksim} \mathcal{N}(0, I_p), ~~~z_i \in \{-1,1 \}, ~~~  \| \theta \|_0\leq s, ~~~i=1, \dots, n.
\end{align}

\noindent This model has been thoroughly investigated and serves as a toy model for the analysis of more complex (sparse) Gaussian mixture models. For instance, Verzelen and Arias-Castro \cite{VerzelenAriasCastro17} study detection limits, Fan et al.~\cite{FanLiuWangYang18} computational detection limits, and Jin et al.~\cite{JinKeWang17} the possibility of consistent clustering. Moreover, in the non-sparse setting of this model, clustering is considered by \cite{Ndaoud19,AbbeFanWang20}  and \cite{WuZhou19,BalakrishnanWainwrightYu17} analyze the performance of the EM algorithm . 
\subsection{A computationally feasible algorithm}
Below the BBP threshold \cite{BaikBenArousPeche05}, i.e., when $p/n> (1+\Omega(1))\|\theta\|^4$, consistent clustering is, in general, information theoretically impossible \cite{JinKeWang17,Ndaoud19}. To circumvent this issue it is necessary to take the additional information that $\theta$ is sparse into account.

We modify the low-dimensional clustering algorithm proposed by Vempala and Wang \cite{VempalaWang04}. Their approach consists of computing the SVD of the data matrix, projecting the data onto the low-dimensional space spanned by the first left singular vector and then running a clustering algorithm such as $k$-means. 

We propose to modify the first step of this algorithm by running a sparse PCA algorithm. In particular, we use a variant of a semidefinite program which was initially proposed by d'Aspremont et al.~\cite{dAspremontElGhaouiJordanLanckriet07} and further developed and analyzed in \cite{AminiWainwright09,VuChoLeiRohe13,LeiVu15}. Other theoretically investigated approaches to solve the sparse PCA problem include diagonal thresholding \cite{JohnstoneLu07}, iterative thresholding \cite{Ma13}, covariance thresholding \cite{KrauthgamerNadlerVilenchik15,DesphandeMontanari16} and axis aligned random projections \cite{GataricWangSamworth20}. 

In spirit, Algorithm \ref{alg 1} is similar to the two-stage selection methods proposed in \cite{AzizyanSinghWasserman13,JinWang16,JinKeWang17}, and we discuss differences below.

 \begin{algorithm}[h] \label{alg 1}
 	\SetAlgoLined
 	\KwIn{Data matrix $X=[X_1, \dots, X_n] \in\mathr^{p\times n}$, tuning parameter $\lambda$}
 	\KwOut{Clustering label vector $\hat z\in \{-1,1\}^n$}
 	\nl Compute estimator for the projector onto first sparse principal component {via SDP}
 	$$
 	\hat P:= \argmax_{P \in \mathcal{F}^1} \left \langle \frac{XX^T}{n}, P \right \rangle - \lambda \|P\|_1,
 	$$~\\ 
 	where $\mathcal{F}^1:=\{P: ~P^T=P, ~\text{tr}(P)=1, ~0 \preceq P \preceq I \}. $ \\
 	\nl Perform an eigendecomposition of $\hat P$ and compute a leading eigenvector $\hat u$. \\
 	\nl Define $\hat Y =\hat u^T X \in\mathr^{ n}$ and return
 	\begin{equation} \hat z_i=\text{sgn}(\hat Y_i). \end{equation}
 	\caption{Sparse spectral clustering}\label{alg:mainsimple}
\end{algorithm}
In the following theorem we show that Algorithm \ref{alg 1} achieves an exponentially-small misclustering error when the squared sparsity is of smaller order than the sample size. We emphasize that Algorithm \ref{alg 1} does not require an impractical sample splitting\footnote{\emph{Sample splitting} would split the data into two copies with independent noise (as in the proof of Theorem~\ref{thm:complowerz}) and use one for each of the steps 1 and 3 of Algorithm~\ref{alg 1}. This would make the analysis easier but yields an algorithm that would be less natural to use in practice.} step. This makes the proof of Theorem \ref{thm mainsimple} more difficult as $\hat u$ and each $X_i$ are not independent. We overcome this difficulty by using the leave-one-out method combined with a careful analysis of the KKT conditions of the SDP estimator $\hat P$. 
\begin{theorem} \label{thm mainsimple}
Assume that  $\log(p) \leq n$ and $\|\theta\|_\infty \leq \kappa \leq \sqrt{\log(p)}$ for some $\kappa > 0$. Moreover, suppose that for some large enough constant $C>0$, $\lambda = C(1+\kappa)\sqrt{\log(p)/n}$ and that for some small enough constant $c>0$ 
	\begin{align} \label{thm mainsimple signal}
\frac{s^2 \log(p)(1+\kappa^2)}{n \|\theta\|^4} =:\tau_n^2 \leq c. 
	\end{align}Then the output of Algorithm \ref{alg 1} satisfies with probability at least $1-8p^{-1}-2e^{-\|\theta\|/2}$ that for another constant $c'>0$,
	\begin{align} \label{thm mainsimple eq rate}
	\ell(\hat z, z):=\min_{\pi \in \{-1,1\}} \frac{1}{n} \sum_{i=1}^n \mathbf{1}(\pi \hat z_i \neq z_i )  \leq 2\exp \left ( - {\frac{\|\theta\|^2}{2}(1-c'\tau_n-\|\theta\|^{-1})} \right ). 
	\end{align}
\end{theorem}
\noindent 
Lu and Zhou \cite{LuZhou16} prove a minimax lower bound in the analogous setting without sparsity. Their result implies the following minimax lower bound in our setting (even if $\theta$ is known to the algorithm):
\[ \inf_{\hat z} \sup_{\substack{(\theta, z): ~\|\theta\| \geq \Delta, ~\|\theta\|_0\leq s, \\ \|\theta\|_\infty \leq \kappa, ~z \in \{-1,1\}^n}} \mathbb{E}_{(\theta, z)}\, \ell(\hat z, z) \geq \exp \left ( -\frac{\Delta^2}{2} \big (1-o(1)\big ) \right ) ~~\text{as} ~\Delta \rightarrow \infty. \]
Hence, when $\|\theta\|=\omega(1)$ and $s^2/n=o(\|\theta\|^4/(\log(p)(1+\kappa^2)))$, the convergence rate in \eqref{thm mainsimple eq rate} is minimax optimal. It was previously shown by \cite{LuZhou16,Ndaoud19} that this rate is achievable when $p/n = o(\|\theta\|^2)$ (which is above the BBP transition \cite{BaikBenArousPeche05}, where the sparsity assumption is not needed). Algorithm \ref{alg:mainsimple} is the first clustering procedure to provably achieve this minimax optimal misclustering error in the regime where sparsity must be exploited.

When $p/n=o(\|\theta\|^4$) and $p/n=\omega(\|\theta\|^2)$ only a slower convergence rate is achievable in the setting without sparsity \cite{GiraudVerzelen19,Ndaoud19,AbbeFanWang20}. This corresponds to a regime where it is impossible to consistently estimate the direction of $\theta$, but where consistent clustering is possible. By contrast, such a regime does not exist in the sparse high-dimensional setting when  $p/n=\omega(\|\theta\|^4)$. This is due to the fact that is necessary to exploit the sparsity of $\theta$ and  estimate its direction in the clustering process \cite{JinKeWang17}.

We now discuss existing theoretical results in high-dimensional sparse clustering. The articles \cite{AzizyanSinghWasserman13,AzizyanSinghWasserman15,CaiMaZhang19} focus on estimating classification rules from unlabeled data and study the risk of misclassifying a new observation. This risk measure is easier to analyze because the classification rule is independent of the new observation. Moreover, their setting is slightly different. For instance, Cai et al.~\cite{CaiMaZhang19} assume that $\|\theta\|=O(1)$ and consider the more general case with arbitrary covariance matrix $\Sigma$ and sparsity of $\Sigma^{-1}\theta$. They propose a sparse high-dimensional EM-algorithm and prove sharp bounds on the excess prediction risk, provided that they have a sufficiently good initializer. They (and Azizyan et al.~\cite{AzizyanSinghWasserman15}, too) obtain this initializer by using the Hardt-Price algorithm \cite{HardtPrice15} and penalized estimation. Ultimately this requires that $s^{12}/n=o(1)$. 

Similarly to the present article, \cite{AzizyanSinghWasserman13,JinWang16,JinKeWang17} propose two-stage algorithms, selecting first the relevant coordinates and then performing clustering. In \cite{JinWang16} and \cite{JinKeWang17}, suboptimal polynomial rates of convergence are shown for the misclustering error. In particular, Jin et al.~\cite{JinKeWang17} assume that $\theta_j \overset{i.i.d.}{\thicksim} (1-\epsilon)\delta_0+\epsilon \delta_\kappa/2+\epsilon \delta_{-\kappa}/2$ where $\kappa=o(1), \epsilon=o(1)$ and give precise bounds for which regimes of $(p,n,\kappa,\epsilon)$ consistent clustering is possible and for which it is not. Most relevant to the present work, they prove, ignoring log-factors, that  $\|\theta\|^2 =\omega(1+s/n)$ is a necessary and sufficient condition for consistent clustering. However, the algorithm that achieves this performance bound is based on exhaustive search which is not computationally efficient. By contrast, ignoring log-factors, we require in addition that  $\|\theta\|^2=\omega(1+s/\sqrt{n})$ for consistent recovery when $\theta$ is sampled from their prior, matching the requirements of a polynomial-time algorithm they propose. Observing this discrepancy between the performance of their polynomial-time algorithm and the exhaustive search algorithm, they conjecture computational gaps, for which we give evidence in Section~\ref{section:lower}.

In contrast to the above work, our results apply for arbitrary $\theta$ and $z$ from a given parameter parameter space (instead of a specific prior), we do not assume $\kappa=o(1)$, and we achieve optimal exponential convergence rates.

The assumption $\kappa=O(\sqrt{\log p})$ is less restrictive than typical assumptions on $\theta$ used in the literature. For instance, \cite{JinKeWang17,FanLiuWangYang18} both assume that $\theta \in \{-\kappa,\kappa,0\}^p$ and \cite{JinWang16,JinKeWang17} assume that $\kappa=o(1)$. 

We show next that in fact no assumption on $\|\theta\|_\infty$ is needed at all: if $\|\theta\|_\infty \geq \kappa$ for some large enough constant $\kappa > 0$, then clustering becomes much easier.  In particular,
it is possible to achieve exponential convergence rates in $\|\theta\|^2$ as soon as $s/n =o(\|\theta\|^2/\log(ep/s))$. Hence, in this regime no computational gap exists.
\begin{theorem} \label{theorem upper 2}
Suppose that $\log(p)\leq n$ and that $\|\theta\|_\infty \geq \kappa$ for some $\kappa > 0$ 
such that for some small enough constant $c>0$,
\begin{align*}
     \frac{1}{\kappa^4} + \frac{s \log(ep/s)}{n\|\theta\|^2} \leq \tilde \tau_n^2 \leq c.
\end{align*}
Then there exists a polynomial time algorithm with output $\hat z$ and a constant $c'>0$ such that with probability at least $1-(3s)/p-2e^{-\|\theta\|/4}$,
\begin{align}
    \ell(\hat z, z ) \leq 2 \exp\left ( -  \frac{\|\theta\|^2 }{4} \left (1-c'\tilde \tau_n-\|\theta\|^{-1} \right ) \right ). \label{upper large kappa}
\end{align}
\end{theorem}
\noindent 

\subsection{Computational lower bounds} \label{section:lower}

Jin et. al. \cite{JinKeWang17} conjecture a computational gap. In particular, in analogy to sparse PCA, they suggest that a polynomial-time algorithm can have expected misclustering error better than $1/2$ in the regime $p/n=\omega( \|\theta\|^4)$ (below BBP) only if the additional condition $\|\theta\|^4 = \Omega(s^2/n)$ is fulfilled. This is nontrivial when $s^2 \geq n$ and suggests that there is a range of parameters where from a statistical point of view consistent estimation is possible, but from a computational perspective it is not.

Starting with the seminal work of Berthet and Rigollet on sparse PCA \cite{BerthetRigollet13a,BerthetRigollet13b} there has been a huge influx of works studying such computational gaps in sparse PCA and related problems. Various forms of rigorous evidence for computational hardness have been proposed, including reductions from the conjectured-hard planted clique problem \cite{BerthetRigollet13a,GaoMaZhou17,BrennanBresler19b,BrennanBresler19,BrennanBresler20}, statistical query lower bounds \cite{Kearns98,sq-clique,DiakonikolasKaneStewart17,FanLiuWangYang18,NilesWeedRigollet19}, sum-of-squares lower bounds \cite{sos-sparse-pca,pcal,HopkinsKothariPotechinRaghavendraSchrammSteurer17}, and analysis of the low-degree likelihood ratio \cite{HopkinsSteurer17,HopkinsKothariPotechinRaghavendraSchrammSteurer17,Hopkins18,KuniskyWeinBandeira19,DingKuniskyWeinBandeira19}. 

In the Gaussian mixture model we consider, Fan et al. \cite{FanLiuWangYang18} have shown that in the statistical query model, testing
\begin{align} \label{eq testing problem}  H_0:~\theta=0~~~~~ \text{vs.} ~~~~~H_1:~\|\theta\| \geq \Delta \end{align}
is not possible in polynomial time under certain conditions which include $p/n = \omega( \Delta^4)$, $s^2/n = \omega( \Delta^4)$, and $n = o(p)$. A similar result has been proven when assuming the planted clique conjecture \cite{BrennanBresler19}. In addition, Fan et al. \cite{FanLiuWangYang18} show that if the above testing problem is indeed hard, then as a consequence, no polynomial-time classification rule can near-perfectly match (with error probability $o(1)$) the output of Fisher's linear discriminant (the statistically optimal classifier) when predicting the label of a new observation.

We now complement these results two-fold: first, we give evidence based on the low-degree likelihood ratio that the above testing problem is computationally hard when $p/n\ge (1+\Omega(1))  \Delta^4$ (i.e.\ below BBP) and $s^2/n = \omega( \Delta^4 \log pn)$, including a precise lower bound on the conjectured runtime. Our result covers a wider regime of parameters than prior work: we do not require $n = o(p)$, and capture the sharp BBP transition. We also give a reduction showing that if the testing problem is indeed hard then this implies that even ``weak'' clustering (better than random guessing) cannot be achieved in polynomial time.
Similar detection-to-recovery reductions have been given in other settings (see e.g., Section~5.1 of~\cite{MaWu15} or Section~10 of \cite{BBH}).

We now describe the low-degree framework~\cite{HopkinsSteurer17,HopkinsKothariPotechinRaghavendraSchrammSteurer17,Hopkins18} upon which our first result is based, referring the reader to \cite{Hopkins18,KuniskyWeinBandeira19} for more details. Suppose $\PP_n$ and $\QQ_n$ are probability distributions on $\RR^N$ for some $N = N_n$ (where $n$ is a natural notion of problem size or dimension). We will be interested in how well a (multivariate) low-degree polynomial $f:\RR^N \to \RR$ can distinguish $\PP_n$ from $\QQ_n$ in the sense that $f$ outputs a large value when the input is drawn from $\PP_n$ and a small value when the input is drawn from $\QQ_n$. Specifically, we will be interested in the quantity
\begin{equation}\label{eq:L-defn}
\|L_n^{\le D}\| := \max_{f\text{ deg } \le D} \frac{\E_{X \sim \PP_n}[f(X)]}{\sqrt{\E_{X \sim \QQ_n}[f(X)^2]}}
\end{equation}
where the maximization is over polynomials $f: \RR^N \to \RR$ of degree (at most) $D$. (The notation $\|L_n^{\le D}\|$ comes from the fact that an equivalent characterization of this value is the $L^2(\QQ_n)$-norm of the low-degree likelihood ratio $L_n^{\le D}$, which is the orthogonal projection of the likelihood ratio $L_n = \frac{d\PP_n}{d\QQ_n}$ onto the subspace of degree-$D$ polynomials \cite{HopkinsSteurer17,HopkinsKothariPotechinRaghavendraSchrammSteurer17,Hopkins18}.) We think of $\|L_n^{\le D}\|$ as an informal measure of success for degree-$D$ polynomials: if $\|L_n^{\le D}\| \to \infty$ as $n \to \infty$, this suggests (but does not rigorously imply) that $\PP_n$ and $\QQ_n$ can be consistently distinguished (i.e., with both type I and type II errors tending to $0$ as $n \to \infty$) by thresholding a degree-$D$ polynomial. On the other hand, if $\|L_n^{\le D}\| = O(1)$, this implies that no degree-$D$ test succeeds in a particular formal sense (see Section~\ref{sec:rig-ld}). Furthermore, for many ``natural'' high-dimensional testing problems---including planted clique, sparse PCA, community detection, tensor PCA, and others---it has been shown that in the ``easy'' parameter regime (where consistent detection is possible in polynomial time) we have $\|L_n^{\le D}\| = \omega(1)$ for some $D = O(\log N)$, while in the conjectured ``hard'' regime we have $\|L_n^{\le D}\| = O(1)$ for some $D = \omega(\log N)$~\cite{HopkinsSteurer17,HopkinsKothariPotechinRaghavendraSchrammSteurer17,Hopkins18,KuniskyWeinBandeira19,DingKuniskyWeinBandeira19}. In other words, $O(\log N)$-degree polynomials are as powerful as the best known polynomial-time algorithms for all of these problems. One explanation for this is that the best known algorithm often takes the form of a spectral method that thresholds the leading eigenvalue of some polynomial-size matrix whose entries are constant-degree polynomials of the input, and such a spectral method can be implemented as an $O(\log N)$-degree polynomial via power iteration (under certain mild conditions, including a spectral gap; see Theorem~4.4 of~\cite{KuniskyWeinBandeira19}). In light of the above, $\|L_n^{\le D}\|$ can be used to predict the computational complexity of various testing problem, and to give concrete evidence for hardness.

We now specialize to our setting of interest and define the appropriate distributions $\PP_n$ and $\QQ_n$. The planted distribution $\PP_n$ will consist of samples drawn from the model~\eqref{eq:model}, where $(\theta,z)$ are drawn from a particular prior taking values in
\begin{equation}\label{eq:P-Del}
\mathcal{P}_\Delta=\{ (\theta, z)~: ~z \in \{-1, 1\}^n, ~\|\theta\|_0  \leq s, ~\|\theta\| \geq \Delta,~\theta \in \mathbb{R}^p \}.
\end{equation}
The null distribution $\QQ_n$ will consist of i.i.d.\ Gaussian samples with no planted signal.

\begin{definition}\label{def:PQ}
We index the distributions $\PP_n$ and $\QQ_n$ by the sample size $n$, and allow the other parameters to scale with $n$: $p = p_n$, $\Delta = \Delta_n$, $s = s_n$. We define $\PP_n$ and $\QQ_n$ to be the following distributions over $\RR^{pn}$. Under $\PP_n$, first draw $(\theta,z)$ from the following prior and then observe $n$ samples from the model~\eqref{eq:model}. The prior samples $z$ and $\theta$ independently as follows:
\begin{align*}
z_i \overset{i.i.d.}{\thicksim} \mathcal{R}
\end{align*}
where $\mathcal{R}$ denotes the Rademacher distribution, and
\begin{align*}
& S \overset{\text{unif.}}{\thicksim} \{ S \subset \{1, \dots, p\}, ~|S|=s \} \\ 
& 	\theta_i | S \overset{independently}{\thicksim}  \begin{cases} \frac{\Delta}{\sqrt{s}} \mathcal{R} & i \in S \\
0 & \text{else}. \end{cases}
\end{align*}
Under $\QQ_n$, we set $\theta = 0$ (which yields $z$ immaterial) and again observe $n$ samples from the model~\eqref{eq:model}.
\end{definition}


\noindent In accordance with the above discussion, the early works on the low-degree framework~\cite{HopkinsSteurer17,HopkinsKothariPotechinRaghavendraSchrammSteurer17,Hopkins18} conjectured (either informally or formally) that for a broad class of high-dimensional testing problems, $\|L_n^{\le D}\| = O(1)$ implies failure of all polynomial-time algorithms---we refer to this idea as the \emph{low-degree conjecture}. In particular, \cite{Hopkins18} stated a formal conjecture (see Conjecture~2.2.4 of~\cite{Hopkins18}) which describes a precise class of testing problems for which the low-degree conjecture is believed to hold. While our testing problem of interest (Definition~\ref{def:PQ}) does not quite fall into the class described by~\cite{Hopkins18} (due to the specific type of symmetry in our distributions), we believe that our testing problem lies solidly within the class of problems for which the low-degree conjecture is generally believed to hold. We therefore state a formal version of the low-degree conjecture for our specific testing problem.

\begin{conjecture}[Low-degree conjecture]
\label{conj:lowdeg}
Let $\PP_n$ and $\QQ_n$ be defined as in Definition~\ref{def:PQ}, for some choice of $p = p_n$, $\Delta = \Delta_n$, $s = s_n$. Let $L_n = \frac{d\PP_n}{d\QQ_n}$ and define $\|L_n^{\le D}\|$ accordingly as in~\eqref{eq:L-defn}. For $\delta > 0$, let $\PP_n^{(\delta)}$ be defined the same as $\PP_n$ but with $(1-\delta)\Delta$ in place of $\Delta$.
\begin{itemize}
    \item If there exists $D = D_n$ satisfying $D = \omega(\log pn)$ and $\|L_n^{\le D}\| = O(1)$ then for any fixed $\delta > 0$, $\PP_n^{(\delta)}$ and $\QQ_n$ cannot be \emph{consistently} distinguished in polynomial time, i.e., there is no polynomial-time test $t_n: \RR^{pn} \to \{0,1\}$ satisfying
   $$ \mathbb{E} [t_n(X)|X \sim \PP_n^{(\delta)}]+\mathbb{E} [1-t_n(X)|X \sim \QQ_n] =o(1).$$
    \item If there exists $D = D_n$ satisfying $D = \omega(\log pn)$ and $\|L_n^{\le D}\| = 1+o(1)$ then for any fixed $\delta > 0$, $\PP_n^{(\delta)}$ and $\QQ_n$ cannot be \emph{weakly} distinguished in polynomial time, i.e., there is no polynomial-time test $t_n: \RR^{pn} \to \{0,1\}$ satisfying
    $$ \mathbb{E} [t_n(X)|X \sim \PP_n^{(\delta)}]+\mathbb{E} [1-t_n(X)|X \sim \QQ_n] = 1-\Omega(1).$$
\end{itemize}
\end{conjecture}

\noindent Conjecture~\ref{conj:lowdeg} is a refined version of existing ideas~\cite{HopkinsSteurer17,HopkinsKothariPotechinRaghavendraSchrammSteurer17,Hopkins18} but has not appeared in this precise form before (even for other distributions $\PP_n$ and $\QQ_n$). For instance, the second statement regarding weak testing has not (to our knowledge) been stated before, although it is a natural extension of the first statement. In Section~\ref{sec:rig-ld} we provide justification for both statements in the conjecture by showing that bounds on $\|L_n^{\le D}\|$ rigorously imply failure of polynomial-based tests (defined appropriately).
The choice of super-logarithmic degree $D = \omega(\log pn)$ is justified by the belief that if a polynomial-time test exists then there should also exist a spectral method (suitably defined) that succeeds, and any such spectral method can be implemented as a $O(\log pn)$-degree polynomial (see Theorem~4.4 of \cite{KuniskyWeinBandeira19}).
We remark that the conclusion of Conjecture~\ref{conj:lowdeg} pertains not to $\PP_n$ but to $\PP_n^{(\delta)}$ which has a slightly-reduced signal-to-noise ratio. We have included this in accordance with the ``noise operator'' used in the formal conjecture of~\cite{Hopkins18} (and refined by~\cite{lowdeg-counter}). The presence of $\delta$ has no substantial effect on the conclusions we can draw from Conjecture~\ref{conj:lowdeg} but may be important for the conjecture to be true, e.g., in situations very close to the critical value for the BBP transition. While we will use Conjecture~\ref{conj:lowdeg} to guide our discussions, we emphasize that the reader may prefer to think of our results as \emph{unconditional} lower bounds against a powerful class of algorithms (namely low-degree polynomials); see Section~\ref{sec:rig-ld}.

We now state our low-degree lower bound for our testing problem of interest.

\begin{theorem} \label{thm:complowertheta}
	In the setting of Definition~\ref{def:PQ}, if
	\begin{equation}\label{eq:lowdeg-cond}
	\limsup_n \left [\sqrt{\frac{n \Delta^4 }{p}}+ \sqrt{\frac{4n \Delta^4 D}{s^2}} \right ] < 1,
	\end{equation}
then
	\[ \| L^{\leq D}_n \|^2 =1+O \left ( {\frac{n \Delta^4 }{p}}+ {\frac{n \Delta^4 D}{s^2}} \right ) . \]
\end{theorem}

\noindent Our primary regime of interest is $p/n \ge (1+\Omega(1))  \Delta^4$, i.e., below the BBP transition \cite{BaikBenArousPeche05}. Otherwise the sparsity assumption is not needed and a polynomial-time test based on the top eigenvalue of $XX^T$ solves the testing problem \eqref{eq testing problem} \cite{BaikBenArousPeche05,VerzelenAriasCastro17}, and clustering better than random guess in polynomial time is provably possible \cite{Ndaoud19}. In this case, to satisfy \eqref{eq:lowdeg-cond}, it is sufficient to have $s^2/(n \Delta^4) = \omega(D)$. Thus, Conjecture~\ref{conj:lowdeg} posits that if $s^2/n = \omega(\Delta^4 \log pn)$ then consistent testing cannot be achieved in polynomial time.
A more general version of the low-degree conjecture (Hypothesis~2.1.5 of~\cite{Hopkins18}; see also~\cite{KuniskyWeinBandeira19,DingKuniskyWeinBandeira19}) gives a conjectured relationship between degree and runtime: if $\|L_n^{\le D}\| = O(1)$ for some $D = D_n$, then any test achieving strong detection requires runtime $\exp(\tilde\Omega(D))$ where $\tilde\Omega(\cdot)$ hides factors of $\log(pn)$. With this in mind, our results suggest that runtime $\exp(\tilde\Omega(s^2/(n \Delta^4)))$ is required for our problem of interest. (This is essentially tight due to the subexponential-time algorithms for sparse PCA~\cite{DingKuniskyWeinBandeira19,anytime-pca}.) In contrast, reductions from planted clique only suggest a lower bound of $n^{O(\log n)}$ on the runtime, because this runtime is sufficient to solve planted clique.

A related analysis of the low-degree likelihood ratio has been given in the setting of sparse PCA \cite{DingKuniskyWeinBandeira19}. In contrast, our approach is simpler, avoiding direct moment calculations by using Bernstein's inequality.

Our next result is a reduction showing that if the above testing problem is indeed hard (as suggested by Theorem~\ref{thm:complowertheta}) then clustering better than with a random guess is at least as hard. Recall the definition of $\mathcal{P}_\Delta$ from~\eqref{eq:P-Del}.

\begin{theorem} \label{thm:complowerz}
Suppose there exists a clustering algorithm $\hat z$ with runtime $O(r_n)$ such that for some misclustering tolerance $0 \leq \delta < 1/2$ and some error probability $0 \leq  \alpha \leq 1$,
\begin{align}
\sup_{(\theta, z) \in \mathcal{P}_\Delta}\mathbb{P}_{(\theta,z)}\left ( \ell(\hat z, z) > \delta \right ) \le \alpha  \label{thm:complowerz power},
\end{align}
for some
\begin{align} \label{eq complower Delta} \Delta > 2\sqrt{6}(1-2\delta)^{-1} \sqrt{1+\epsilon^{-2}} \sqrt{s \log(ep/s)/n}
\end{align}
and where $\epsilon \in (0,1]$. 
Then there exists a test $t_n: \RR^{pn} \to \{0,1\}$ with runtime $O(r_n+pn)$ and error probability satisfying 
\begin{align}
\mathbb{E}_{0} [t_n]+\sup_{((1-\epsilon^2)\theta,z) \in \mathcal{P}_\Delta } \mathbb{E}_{(\theta,z)}[1-t_n] \leq \alpha +s/p.
\end{align}
\end{theorem}
\noindent
The condition~\eqref{eq complower Delta} is essentially the condition under which the detection problem is information-theoretically possible \cite{VerzelenAriasCastro17} when $\Delta \to \infty$. This condition is not restrictive in our setting: since our aim is to show hardness whenever $s^2/n = \omega( \Delta^4 \log pn)$, it is sufficient to restrict to the ``boundary'' case, say, $s^2/n = o( \Delta^4 \log^2 pn)$, in which case~\eqref{eq complower Delta} is satisfied under the very mild condition $\log(pn) \le n^{1/4}$ (provided $\delta$ is bounded away from $1/2$ and $\epsilon$ is bounded away from $0$).

Combining Theorems \ref{thm:complowertheta} and \ref{thm:complowerz}, and assuming the low degree conjecture (Conjecture~\ref{conj:lowdeg}), we can conclude the following. Below the BBP transition, i.e., when $p /n\ge (1+\Omega(1)) \Delta^4$, if $s^2/n = \omega(\Delta^4 \log pn)$ and~\eqref{eq complower Delta} holds then any polynomial-time algorithm $\hat z$ fails with an asymptotically strictly positive probability in the sense that for any fixed $0 \le \delta < 1/2$,
\begin{align*}
\sup_{(\theta, z) \in \mathcal{P}_\Delta} 
\mathbb{P}_{(\theta,z)} (\ell(\hat z, z)> \delta ) =\Omega(1).
\end{align*}
Moreover, if we additionally assume that we are strictly below the BBP transition in the sense that $p/n = \omega(\Delta^4)$, we can conclude the stronger statement that any polynomial-time algorithm $\hat z$ fails with probability approaching one, i.e., for any fixed $0 \leq \delta < 1/2$,
\begin{align*}
\sup_{(\theta, z) \in \mathcal{P}_\Delta} 
\mathbb{P}_{(\theta,z)} (\ell(\hat z, z) > \delta ) =1-o(1).
\end{align*}
Up to logarithmic factors these bounds match the conjecture by Jin et al. \cite{JinKeWang17}.
Hence Algorithm \ref{alg:mainsimple} is, up to the logarithmic factor in $p$, computationally optimal. 
Deshpande and Montanari \cite{DesphandeMontanari16} show how to remove the $\sqrt{\log p}$-factor in sparse PCA using the \emph{covariance thresholding} algorithm. We leave it open as an interesting question how to do the same in our clustering situation.

\subsection{Rigorous low-degree lower bounds}
\label{sec:rig-ld}

In this section we justify Conjecture~\ref{conj:lowdeg} by discussing the sense in which bounds on $\|L_n^{\le D}\|$ rigorously imply failure of polynomial-based tests. A few results along these lines exist already for various notions of what it means for a polynomial to succeed at testing. Perhaps the most basic is to define ``success'' of $f: \RR^N \to \RR$ by
\begin{equation}\label{eq:f-success-basic}
\Ex_{X \sim \QQ_n}[f(X)] = 0, \; \Ex_{X \sim \PP_n}[f(X)] = 1, \; \Var_{X \sim \QQ_n}[f(X)] = o(1), \; \Var_{X \sim \PP_n}[f(X)] = o(1).
\end{equation}
Clearly if~\eqref{eq:f-success-basic} holds then by Chebyshev's inequality, consistent detection (i.e., with error probability $o(1)$) is possible by thresholding $f$. On the other hand, if $\|L_n^{\le D}\| = O(1)$ then it is immediate from the definition~\eqref{eq:L-defn} that no degree-$D$ polynomial $f$ can satisfy~\eqref{eq:f-success-basic} (in fact, even the first three conditions of~\eqref{eq:f-success-basic} cannot be simultaneously satisfied). For a slightly different notion of ``success'', other results of this flavor have been given in~\cite{KuniskyWeinBandeira19}, showing that $\|L_n^{\le D}\| = O(1)$ implies failure of polynomial-based tests and spectral methods (see Theorems~4.3 and 4.4 of~\cite{KuniskyWeinBandeira19}).

Our contribution in this section is to define a new notion of ``success'' for polynomial-based tests and show that bounds on $\|L_n^{\le D}\|$ imply failure of all degree-$D$ tests in this sense. Notably, this framework addresses not only the case of consistent testing (discussed above) but also the case of \emph{weak} testing (i.e., beating a random guess) which has not been explored by the existing results of this type. As a result, we justify the second statement in Conjecture~\ref{conj:lowdeg}.

We consider a general setting of two distributions $\PP$ and $\QQ$ over data $X \in \RR^N$, and we consider the testing problem
\[ H_0:~X \thicksim \QQ ~~~~~~~{\text{against}} ~~~~~H_1:~X\thicksim \PP. \]
We let $L = \frac{d\PP}{d\QQ}$ denote the associated likelihood ratio. We will sometimes consider an asymptotic regime, in which case $\PP = \PP_n$, $\QQ = \QQ_n$, and $N = N_n$. Given a polynomial $f: \RR^N \to \RR$, we consider the randomized test
\begin{align*}
    \Psi_f(X):=\begin{cases} B(f(X)) & \text{if } f(X) \in[0,1] \\
    \perp & \text{otherwise}
    \end{cases}
\end{align*}
where $B(f(X))$ denotes a Bernoulli random variable with success probability $f(X)$. 
We say that $\Psi_f$ \emph{fails} if it outputs $1$ under $H_0$, outputs $0$ under $H_1$, or outputs the symbol $\perp$.

We define a subclass of randomized degree-$D$ tests which outputs the symbol~$\perp$ with probability at most $\tau$ as follows
\begin{align*}
    \mathcal{T}_{\tau, D,K}:= \{ \Psi=\Psi_f ~:~ \text{deg}(f)\leq D, ~&\max (\mathbb{E}_\QQ f^2, \mathbb{E}_\PP f^2) \leq K,\\
    &\max (\QQ(f \notin [0,1]),\PP(f \notin [0,1]) ) \leq \tau \}.
\end{align*}
For a good test, we imagine that the parameters scale asymptotically as $\tau_n = o(1)$ and $K_n = O(1)$. The degree $D_n = D$ can have any dependence on $n$. The following result shows impossibility of weak detection when $\|L_n^{\le D}\| = 1 + o(1)$.


\begin{lemma}\label{lem:ld-weak}
For any low-degree test $\Psi \in \mathcal{T}_{\tau,D,K}$ we have that 
\[ \QQ( \Psi ~ \text{fails}) + \PP(\Psi ~\text{fails}) \geq 1-\sqrt{K} \sqrt{ \|L^{\leq D}\|^2-1}-(\sqrt{\tau}+2\sqrt{K})\sqrt{\tau}. \]
In particular, if $\|L_n^{\le D}\| = 1 + o(1)$, $\tau_n = o(1)$, and $K_n = O(1)$ then
\[ \QQ_n( \Psi ~ \text{fails}) + \PP_n(\Psi ~\text{fails}) \geq 1-o(1). \]
\end{lemma}

\begin{proof}
Recall that $L^{\le D}$ denotes the likelihood ratio $L = \frac{d\PP}{d\QQ}$ projected onto the subspace of degree-$D$ polynomials, where projection is orthogonal with respect to the $L^2(\QQ)$ inner product $\langle f,g \rangle := \mathbb{E}_{X \sim \QQ}[f(X)g(X)]$. Also, $\|L_n^{\le D}\|$ denotes the $L^2(\QQ)$ norm $\|L_n^{\le D}\|^2 := \mathbb{E}_\QQ (L^{\le D})^2$. Below, we will make use of the following fact: if $f$ has degree $D$ then $\mathbb{E}_\QQ f L^{\le D} = \mathbb{E}_\QQ f L$. In particular, $\mathbb{E}_\QQ L^{\le D} = \mathbb{E}_\QQ L = 1$.
We have that 
\begin{align*}
\QQ & ( \Psi ~ \text{fails}) + \PP(\Psi ~\text{fails})   \\
&= \mathbb{E}_\QQ f\mathbf{1}(f \in [0,1])+ \mathbb{E}_\PP (1-f)\mathbf{1}(f \in [0,1]) + \QQ(f \notin [0,1]) + \PP(f \notin [0,1]) \\
&= \mathbb{E}_\QQ f +\mathbb{E}_\QQ(1-f)\mathbf{1}(f \notin [0,1]) + 1 - \mathbb{E}_\PP f +\mathbb{E}_\PP f \mathbf{1}(f \notin [0,1])\\
&= 1-\mathbb{E}_\QQ f (L-1) +\mathbb{E}_\QQ(1-f)\mathbf{1}(f \notin [0,1]) +\mathbb{E}_\PP f \mathbf{1}(f \notin [0,1]) \\
&= 1-\mathbb{E}_\QQ f (L^{\leq D}-1) +\mathbb{E}_\QQ(1-f)\mathbf{1}(f \notin [0,1]) +\mathbb{E}_\PP f \mathbf{1}(f \notin [0,1]) \\
&\ge 1-\mathbb{E}_\QQ f (L^{\leq D}-1) - (\tau + \sqrt{K\tau}) - \sqrt{K\tau} \\ 
&\ge 1-\sqrt{\mathbb{E}_\QQ f^2 \mathbb{E}_\QQ (L^{\leq D}-1)^2}-(\sqrt{\tau}+2\sqrt{K})\sqrt{\tau} \\
&\ge 1-\sqrt{K} \sqrt{ \|L^{\leq D}\|^2-1}-(\sqrt{\tau}+2\sqrt{K})\sqrt{\tau}.
\end{align*}
The second claim in Lemma~\ref{lem:ld-weak} is immediate from the first.
\end{proof}

Next, we rule out consistent detection under the weaker assumption $\|L_n^{\leq D}\|^2 =O(1)$ for $\Psi$ that fulfills in addition the moment condition $\mathbb{E}_\QQ f^2\mathbf{1}(f \notin [0,1]) = o(1)$. Note that by H\"older's inequality, this moment condition holds if $\tau_n = o(1)$ and $\mathbb{E} f^{2+\eta} = O(1)$ for a constant $\eta > 0$.

\begin{lemma}\label{lem:ld-strong}
For any $\Psi = \Psi_f \in \mathcal{T}_{\tau, D,K}$, let $\delta := \QQ(\Psi~\text{fails})$ and $\epsilon := \mathbb{E}_\QQ f^2 \mathbf{1}(f \notin [0,1])$. We have
\[ \mathbb{P}(\Psi~ \text{fails})\geq 1 - \sqrt{\delta+\epsilon}\, \|L^{\leq D}\| - \sqrt{K\tau}. \]
In particular, if $\|L_n^{\le D}\| = O(1)$ and $K_n = O(1)$ then no $\Psi_f \in \mathcal{T}_{\tau, D,K}$ satisfies both $\mathbb{E}_\QQ f^2 \mathbf{1}(f \notin [0,1]) = o(1)$ and
\begin{equation}\label{eq:strong-det}
    \QQ_n \left ( \Psi~ \text{fails}\right ) + \PP_n \left ( \Psi ~\text{fails} \right ) = o(1).
\end{equation}
\end{lemma}
\begin{proof}

We have
\begin{align*}
    \mathbb{E}_\QQ f^2 &= \mathbb{E}_\QQ f^2 \mathbf{1}(f \in [0,1])+ \mathbb{E}_\QQ f^2 \mathbf{1}(f \notin [0,1]) \\
    &\leq \mathbb{E}_\QQ f \mathbf{1}(f \in [0,1])+ \mathbb{E}_\QQ f^2 \mathbf{1}(f \notin [0,1]) \\
    & \leq \QQ(\Psi ~\text{fails})+\mathbb{E}_\QQ f^2 \mathbf{1}(f \notin [0,1]) \\
    & \leq \delta+\epsilon. 
\end{align*}
Hence, we obtain that
\begin{align*}
    \PP({\Psi=1}) &  = \mathbb{E}_\PP f \mathbf{1}(f \in [0,1]) \\
    &= \mathbb{E}_\QQ f L^{\leq D} -\mathbb{E}_\PP f \mathbf{1}(f \notin [0,1]) \\
   &\leq \sqrt{ \mathbb{E}_\QQ f^2} \, \|L^{\leq D}\| + \sqrt{\PP(f \notin [0,1])\mathbb{E}_\PP f^2 } \\
   &\leq \sqrt{\delta+\epsilon}\, \|L^{\leq D}\|+\sqrt{K\tau}
\end{align*}
and thus $\mathbb{P}(\Psi~ \text{fails})\geq 1 - \sqrt{\delta+\epsilon}\, \|L^{\leq D}\| - \sqrt{K\tau}$. The second claim in Lemma~\ref{lem:ld-strong} follows from the first, noting that~\eqref{eq:strong-det} implies $\tau = o(1)$ and $\delta = o(1)$.
\end{proof}

\section{Proofs}
\subsection{Proof of Theorem \ref{thm mainsimple}}
We first note that the estimator $\hat P$ computed in step 1 of Algorithm  \ref{alg:mainsimple} fulfills the following identity
$$ \hat P= \argmax_{P \in \mathcal{F}^1} \left \langle \frac{XX^T}{n}, P \right \rangle - \lambda \|P\|_1 = \argmax_{P \in \mathcal{F}^1} \left \langle \frac{XX^T}{n}- I_p, P \right \rangle - \lambda \|P\|_1,$$
which follows from the trace-constraint in $\mathcal{F}^1$. Henceforth we will always work with the last representation of $\hat P$.  
We define $P=\frac{\theta \theta^T}{\|\theta\|^2}$ and start with the following preliminary lemma which applies the results of \cite{VuChoLeiRohe13} and \cite{LeiVu15} to obtain rates for $\| \hat P-P\|_F$ and false positive control.  For completeness we provide a proof in the appendix. 
\begin{lemma} \label{Lemma rates selec proj}
 Assume that $\log(p) \leq n$ and that $\|\theta\|_\infty \leq \kappa$ for some $\kappa >0$. Moreover, suppose 
that for some large enough constant $C>0$, $\lambda = C(1+\kappa)\sqrt{ \log(p)/n}$ and that $\|\theta\|^2 \geq 2 \lambda s$. Denote $\text{supp}(P)=S$ and $\text{supp}(\hat P)=\hat  S$. Then, with probability at least $1-2p^{-2}$ we have that $\hat S \subset S$ and 
	\begin{align}
	&\|\hat P-P\|_F^2 \lesssim \frac{s^2\log(p)(1+\kappa^2) }{n \|\theta\|^4}~~\text{and} ~~  \left \| \frac{XX^T}{n}-I_p-\theta \theta^T \right \|_\infty \leq \lambda .
	\end{align}
\end{lemma}
Next we prove a key-technical lemma for the deviation of the leave-one-out estimator $\hat P^{(i)}$ which is defined analogous to $\hat P$ with $X_i$ being replaced by an independent copy of it, $X_i'=z_i\theta+\varepsilon_i'$.
Denote $X^{(i)}=[X_1, \dots, X_{i-1}, X_i', X_{i+1}, \dots, X_n]$. The proof uses the KKT-conditions and, crucially, false positive control of $\hat P$ and $\hat P^{(i)}$. 
\begin{lemma} \label{Lemma leave one out}
Assume the conditions of Lemma \ref{Lemma rates selec proj} and set  $\lambda = C \sqrt{\log(p)/n}(1+\kappa)$ for the same $C$ and $\kappa$ as in Lemma \ref{Lemma rates selec proj}. Suppose that $\|\theta \|^2 \geq 8 \lambda s$ and  denote by $\hat P^{(i)}$ the leave-one-out estimator defined analogous to $\hat P$ with $X$ replaced by $X^{(i)}$. Then we have with probability at least $1-6p^{-2}$ that 
	\begin{align}
	\| \hat P-\hat P^{(i)}\|_F \lesssim
\sqrt{	\frac{\log(p)+\kappa^2}{n}} \sqrt{\frac{s^2\log(p)}{n\|\theta\|^4}}.
	\end{align}
	Moreover, on the same event, $\hat P$ and $\hat P^{(i)}$ are $rank$ one projection matrices. 
\end{lemma}
\begin{proof}
We denote \begin{align*}
	\hat M= \frac{XX^T}{n}-I_p,~~~\hat M^{(i)}=\frac{X^{(i)}(X^{(i)})^T}{n}-I_p~~~\text{and}~~~M=\mathbb{E}\hat M=\theta \theta^T.
	\end{align*}
From now on we work on the event where the results of Lemma \ref{Lemma rates selec proj} apply to $\hat P$, $\hat M$, $\hat M^{(i)}$ and $\hat P^{(i)}$, which, by a union bound, occurs with probability at least $1-4p^{-2}$.

	By strong duality we have that
	\begin{align*}
	\max_{P \in \mathcal{F}^1} \langle \hat M, P \rangle - \lambda \|P\|_1 \iff \max_{P \in \mathcal{F}^1} \min_{Z \in B} \langle \hat M - \lambda Z, P \rangle,
	\end{align*}
	where $B:=\{Z: ~Z_{ii}=0, ~Z=Z^T,~\|Z\|_\infty \leq 1\}$ and by the KKT-condition a pair $(\hat P, \hat Z)$ is an optimal solution  if and only if  $\hat Z \in B$ and
	\begin{align*}
	\hat Z_{ij}=\text{sgn}(\hat P_{ij}), ~~(i,j) \in \hat S \times \hat S, ~~~~\hat P = \argmax_{P \in \mathcal{F}^1} \langle \hat M-\lambda \hat Z, P \rangle.\end{align*}
	Suppose now that $(\hat P, \hat Z)$ is an optimal solution. We define a new subgradient $\tilde Z$ by  $\tilde Z_{ij}=\frac{1}{\lambda} \hat M_{ij}$ for $(i,j) \notin S \times S$, $i \neq j$ and $\tilde Z_{ij}=\hat Z_{ij}$ otherwise. We now check that $\tilde Z$ is a valid subgradient and that the pair $(\hat P, \tilde Z)$ also fulfills the KKT-conditions and that it is therefore an optimal solution, too.
	
	For $i=j$ or $(i,j) \in S \times S$ we have by definition of $\tilde Z_{ij}$ that $|\tilde Z_{ij}|=|\hat Z_{ij}|\leq 1$. Moreover, for $i \neq j$ and $(i,j) \notin S \times S$ we have that $\theta_i\theta_j=0$ and hence we obtain
	\begin{align*}
	    |\tilde Z_{ij}|& =\left |\frac{1}{\lambda}\hat M_{ij} \right | = \left |\frac{(XX^T)_{ij}}{\lambda n } \right |  = \left | \frac{(\frac{XX^T}{n}-I_p-\theta \theta^T)_{ij}}{\lambda } \right | \\
	    & \leq \frac{\|\frac{XX^T}{n}-I_p-\theta \theta^T \|_\infty}{ \lambda} \leq 1,
	\end{align*}
where we also used that we work on the event where Lemma \ref{Lemma rates selec proj} applies.
Moreover, since $\hat M$ and $\hat Z$ are symmetric, $\tilde Z$ is symmetric, too. Hence, $\tilde Z$ is a valid subgradient, i.e. $\tilde Z \in B$.

Since we work on the event where Lemma \ref{Lemma rates selec proj} applies, we have that $\hat S \subset S$. Hence, and by definition of $\hat Z$ and $\tilde Z$, we have for $(i,j) \in \hat S \times \hat S$  that 
\begin{align*}
    \text{sgn}(\hat P_{ij})=\hat Z_{ij}=\tilde Z_{ij}. 
\end{align*}
Using again that $\hat S \subset S$ where for $(i,j) \in \hat S \times \hat S$ we have that $\hat Z_{ij}=\tilde Z_{ij}$ we obtain
\begin{align*}
    \langle \hat M-\lambda \tilde Z, \hat P \rangle & = \sum_{(i,j) \in \hat S \times \hat S }(\hat M-\lambda \tilde Z)_{ij} \hat P_{ij} = \sum_{(i,j) \in \hat S \times \hat S }(\hat M-\lambda \hat Z)_{ij} \hat P_{ij} \\ & =   \langle \hat M-\lambda \hat Z, \hat P \rangle
\end{align*}
and hence, by definition of $\hat P$ and $\hat Z$ we have that $\hat P=\argmax_{P \in \mathcal{F}^1}\langle \hat M-\lambda \tilde Z, P \rangle$. Therefore, the KKT-conditions are fulfilled for the pair $(\hat P, \tilde Z)$.

We now show the assertions. Indeed, 
	we have that
	\begin{align*} 
	\hat M-\lambda \tilde Z =  \begin{cases} \hat M_{ij}-\lambda \hat Z_{ij} ~~~(i,j) \in S \times S ,\\
	0~~~~~\text{else}, \end{cases}
	\end{align*}
	and particularly, by Lidski's inequality the first eigenvalue of $\hat M-\lambda \tilde Z$ is lower bounded by 
	\begin{align*}
	\|M\|-\|(\hat M-M-\lambda \tilde Z)_S\| &  \geq \|\theta\|^2-\lambda s-s\|\hat M-M\|_\infty  \\ & \geq \|\theta\|^2-2\lambda s
	> 2\lambda s > 0,
	\end{align*}
	and by the same argument all other eigenvalues of $\hat M$ are upper bounded by $2 \lambda s$. Hence, and since in addition $\mathcal{F}^1$ is the convex hull of its extremal points, the set of rank one projection matrices, $\hat P$ coincides with the projection matrix onto the first eigenspace of $\hat M -\lambda \tilde Z$ (see also Lemma 1 in \cite{LeiVu15}). Likewise, we see that $\hat P^{(i)}$ is also a spectral projector of rank one. Moreover, the first spectral gap of $\hat M-\lambda \tilde Z$ is, by the above reasoning, lower bounded by $\|\theta\|^2-4\lambda s\geq \|\theta\|^2/2$. We now apply the curvature Lemma 3.1 in \cite{VuChoLeiRohe13} to obtain that 
	\begin{align*}
	\|\hat P- \hat P^{(i)}\|_F^2 \leq \frac{2}{\|\theta\|^2-4\lambda s} \left \langle \hat M-\lambda \tilde Z, \hat P-\hat P^{(i)} \right  \rangle.
	\end{align*}
	Moreover, by definition of $\hat P^{(i)}$ we have that,
	\begin{align*}	\left \langle -\hat M^{(i)}, \hat P-\hat P^{(i)} \right  \rangle +\lambda \|\hat P\|_1-\lambda \|\hat P^{(i)}\|_1 \geq  0. \end{align*}
Hence, we obtain that
	\begin{align}
&	\| \hat P-\hat P^{(i)}\|_F^2  \notag  \\\leq & \frac{4}{\|\theta\|^2} \left [ \left \langle  \hat M-\hat M^{(i)}, \hat P-\hat P^{(i)} \right \rangle + \lambda \|\hat P\|_1-\lambda \|\hat P^{(i)}\|_1-\lambda \langle \tilde Z, \hat P-\hat P^{(i)} \rangle \right ]. \label{lemma72proof1}
	\end{align}
To bound the last three terms above, we observe that
\begin{align*}
\langle \tilde Z, \hat P-\hat P^{(i)} \rangle & = \sum_{k,j} \tilde Z_{kj} \hat P_{kj}-\sum_{k,j} \tilde Z_{kj}\hat P^{(i)}_{kj} \\ & = \sum_{(k , j) \in \hat S \times \hat S, ~k \neq j  } \text{sgn}(\hat P_{kj})\hat P_{kj}-\sum_{(k,j) \in \hat S \times \hat S, ~k\neq j} \tilde Z_{kj}\hat P^{(i)}_{kj} \\
& \geq \|\hat P\|_1-\|\hat P^{(i)}\|_1,
\end{align*}
where we used that $\tr(\hat P)=\tr(\hat P^{(i)})=1$, that $\hat P_{ii} \geq 0$ and that $\|\tilde Z\|_\infty \leq 1$. 
Hence the sum of the three last terms in \eqref{lemma72proof1} is bounded above by zero and we estimate
\begin{align*}
\|\hat P-\hat P^{(i)}\|_F^2 & \leq \frac{4}{\|\theta\|^2} \left \langle \hat M-\hat M^{(i)}, \hat P-\hat P^{(i)} \right \rangle \\
& \leq  \frac{4}{\|\theta\|^2} \|\hat M-\hat M^{(i)}\|_\infty \|\hat P-\hat P^{(i)}\|_1 \\
& \leq \frac{8s}{\|\theta\|^2} \left ( \left  \| \frac{\varepsilon_i \varepsilon_{i}^T-\varepsilon_i'(\varepsilon_{i}')^T}{n} \right \|_\infty + \left \| \frac{\theta z_i (\varepsilon_{i}-\varepsilon_{i}')^T}{n} \right \|_\infty  \right)   \|\hat P-P^{(i)}\|_F,
\end{align*}
where we used that on the event we work on we have by  Lemma \ref{Lemma rates selec proj} that $\hat S \subset S$ and $\hat S^{(i)} \subset S$. Moreover, by Gaussian concentration, we have with probability at least $1-p^{-2}$ that  $\| \varepsilon_i \varepsilon_i^T-\varepsilon_i' (\varepsilon_i')^T \|_\infty \leq  4\log(p)$. Applying Gaussian concentration again we bound with probability at least $1-p^{-2}$ 
	\begin{align*}
	\| \left (\theta z_i (\varepsilon_{i}-\varepsilon_{i}')^T \right )\|_\infty \lesssim \kappa \sqrt{\log(p)}. 
	\end{align*}
	Thus, concluding, using another union bound, to incorporate the two events we have on an event of probability at least $1-6p^{-2}$ that 
		\begin{align*}
	\| \hat P-\hat P^{(i)}\|_F &  \lesssim \frac{s\log(p)+s\kappa \sqrt{\log(p)}}{n\|\theta\|^2}. 
	\end{align*}
\end{proof}
\noindent We are now ready to prove Theorem \ref{thm mainsimple}.
\begin{proof}
	Recall that
	\begin{align*}
	\ell (\hat z, z) = \inf_{\pi  \in \{-1,1\}} \frac{1}{n}\sum_{i=1}^n \mathbf{1} ( \pi  \hat z_i \neq z_i). 
	\end{align*}
	Throughout we work on the global event $\Omega$  where the statement of Lemmas \ref{Lemma rates selec proj} and \ref{Lemma leave one out} hold for all $i$, which by a union bound occurs with probability at least $1-8p^{-1}$. We denote this event by $\Omega$. 
Fix one particular $z_i$ and suppose without loss of generality that $z_i=1$. We have that  
\begin{align*}
\mathbf{1} \left ( z_i \neq \pi \hat z_i\right ) = \mathbf{1} \left ( \pi  \hat u^TX_i \leq 0 \right ).
\end{align*}
We define $u=\theta/\|\theta\|$ such that $P=uu^T$. 
Since $\pi$ is either $-1$ or $+1$ for all $i$ we may choose $\pi$ such that $\langle u, \pi \hat u \rangle = \frac{1}{\|\theta\|}\langle \theta , \pi  \hat u \rangle \geq 0$, meaning that we can henceforth assume $\pi=1$ and $\langle u, \hat u \rangle \geq 0$. Hence, we obtain that
\begin{align*}
\mathbf{1} \left ( z_i \neq \pi \hat z_i\right ) \mathbf{1}_{\Omega} & = \mathbf{1} \left ( \hat u^T\theta+ \hat u^T \varepsilon_{i} \leq 0 \right ) \mathbf{1}_{\Omega}\\& \leq   \mathbf{1} \left ( u^T\theta -\|u-\hat u\|\|\theta\|+(\hat u^{(i)})^T\varepsilon_{i}+(\hat u-\hat u^{(i)})^T \varepsilon_{i} \leq 0 \right )\mathbf{1}_{\Omega},
\end{align*}
where we pick $\hat u^{(i)}$ such that $\hat P^{(i)}=\hat u^{(i)} (\hat u^{(i)})^T$ and $\langle \hat u^{(i)}, \hat u \rangle \geq 0$. 
Since $\langle u, \hat u \rangle \geq 0$ we have that 
$$\|\hat u-u\| \mathbf{1}_{\Omega}\leq \|\hat P-P\|_F\mathbf{1}_{\Omega} \lesssim \sqrt{ \frac{ s^2 \log(p)(1+\kappa^2)}{n \|\theta\|^2}} \lesssim \tau_n,$$
per our assumptions. 
Moreover, since on $\Omega$ we have $\text{supp}(\hat P) \subset S$ and $\text{supp}(\hat P^{(i)}) \subset S$, we see that $\hat u$ and $\hat u^{(i)}$ have zero entries at coordinates $i \notin S$. Hence, we obtain
\begin{align*}
|(\hat u-\hat u^{(i)})^T \varepsilon_{i}| \mathbf{1}_{\Omega} =| (\hat u-\hat u^{(i)})_S^T (\varepsilon_{i})_S|\mathbf{1}_{\Omega} \leq \|(\varepsilon_i)_S\| \|\hat u-\hat u^{(i)}\| \mathbf{1}_{\Omega}.
\end{align*}
By Jensen's inequality we have that 
$$\mathbb{E} \|(\varepsilon_i)_S\| \leq \sqrt{s}$$
and using furthermore Borell's inequality (e.g. Theorem 2.2.7 in \cite{GineNickl16}), recognizing that $\|(\varepsilon)_S\| = \sup_{ g \in \mathbb{R}^s:~\|g\|=1} \langle (\varepsilon)_S, g \rangle $, implies that with probability at least $1-e^{-\|\theta\|^2}$
$$\|(\varepsilon_i)_S\| \leq  \sqrt{s}+\sqrt{2}\|\theta\| \lesssim \sqrt{s}(1+\kappa).
$$
Hence, restricting to $\Omega$ where Lemma \ref{Lemma leave one out} holds and since $\langle \hat u, \hat u^{(i)} \rangle \geq 0$ we obtain that with probability at least $1-e^{-\|\theta\|^2}$ 
\begin{align*}
| (\hat u-\hat u^{(i)})^T \varepsilon_{i}| \mathbf{1}_{\Omega} \leq &   \|\hat P-\hat P^{(i)}\|_F \|(\varepsilon_{i})_S\| \mathbf{1}_{\Omega} \\ & \lesssim \sqrt{s}(1+\kappa) \sqrt{\frac{\log(p)+\kappa^2}{n}} \sqrt{\frac{s^2 \log(p)}{n\|\theta\|^4}} \\
& \lesssim \frac{s^2 \log(p) (1+\kappa^2)}{n\|\theta\|^3} =
\tau_n^2 \|\theta\|,
\end{align*}
where we used that $\|\theta\| \leq \sqrt{s}\kappa$ and $\kappa \leq \sqrt{\log(p)}$ by assumption. 
Hence, we have for some constant $c>0$ with probability at least $1-e^{-\|\theta\|^2}$
\begin{align*}
\mathbf{1} \left ( z_i \neq \pi \hat z_i\right ) \mathbf{1}_{\Omega} \leq \mathbf{1} \left ( \|\theta\| (1-c \tau_n) + (\hat u^{(i)})^T \varepsilon_i \leq 0\right ).
\end{align*}
By construction of $\hat u^{(i)}$, $\hat u^{(i)}$ and $\varepsilon_{i}$ are independent and hence $(\hat u^{(i)})^T \varepsilon_i $ is univariate Gaussian (with mean zero and variance one)  distributed. Thus, by a standard tailbound for Gaussian random variables we obtain that 
\begin{align*}
\mathbb{E} \mathbf{1} \left ( \|\theta\| (1-c \tau_n^2) + (\hat u^{(i)})^T \varepsilon_i \leq 0\right ) = \Phi (-\|\theta\| (1-c \tau_n)) \leq e^{-\|\theta\|^2 (1-c \tau_n)^2/2}
\end{align*}
where $\Phi$ denotes the C.D.F.\ of a standard Gaussian random variable. Summarizing, we have, summing over each $i$ (after proper global permutation) that
\begin{align*} \mathbb{E} [\ell(\hat z, z) \mathbf{1}_\Omega ] & \leq \frac{1}{n} \sum_{i=1}^n\mathbb{E} \left [\mathbf{1} \left ( z_i \neq \pi \hat z_i \right ) \mathbf{1}_\Omega \right ] \\ &  \leq e^{-\|\theta\|^2}+e^{-\|\theta\|^2 (1-c \tau_n)^2/2} \leq 2 e^{-\|\theta\|^2 (1-c \tau_n)^2/2}.  \end{align*}
Therefore, applying Markov's inequality and using a union bound to account for $\Omega$, we have with probability at least $1-2e^{-\|\theta\|/2}-8p^{-1}$ that
\begin{align*}
\ell (\hat z, z) \leq 2 e^{-\|\theta\|^2 (1-c\tau_n-\|\theta\|^{-1})/2}.
\end{align*}
\end{proof}
\subsection{Proof of Theorem \ref{thm:complowertheta}}
\begin{proof}
	Let $\pi$ denote the prior on $(\theta,z)$ described in Definition~\ref{def:PQ}. By Theorem 2.6 in \cite{KuniskyWeinBandeira19} we have that
	\begin{align*}
	\| L_n^{\leq D} \|^2 &  = \Ex_{(\theta, z), (\tilde \theta, \tilde z) \sim \pi} \sum_{d=0}^D \frac{1}{d!} \langle z, \tilde z \rangle^d \langle \theta, \tilde \theta \rangle^d \\
	& = 1+\Ex_{(\theta, z), (\tilde \theta, \tilde z) \sim \pi} \sum_{d=1}^{\lfloor D/2 \rfloor} \frac{1}{(2d)!} \langle z, \tilde z \rangle^{2d} \langle \theta, \tilde \theta \rangle^{2d}.
	\end{align*}
	Here, $(\theta,z)$ and $(\tilde \theta, \tilde z)$ are drawn independently from $\pi$. We bound the quantities in the sum above one after the other. Observe that $\langle z, \tilde z \rangle$ is a sum of $n$ i.i.d. Rademacher random variables $R_i, ~i=1, \dots, n$.
	We argue as in the proof of the Khinchine inequality. Indeed, denoting by $g_1, \dots ,g_n$ i.i.d.\ standard Gaussians we have by Jensen's inequality for any even $d_i \geq 2$ that $\mathbb{E} g_i^{d_i} \geq (\mathbb{E} g_i^2)^{d_i/2}=1=\mathbb{E} R_i^{d_i}$. Hence, we obtain that
	\begin{align*}\mathbb{E} \left ( \sum_{i=1}^n R_i\right )^{2d}  \notag & =  \sum_{2d_1+\dots+ 2d_n=2d} \mathbb{E} R_1^{2d_1}  \dots  R_n^{2d_n} \leq \sum_{2d_1+\dots+ 2d_n=2d} \mathbb{E} g_1^{2d_1}  \dots g_n^{2d_n} \\ &  = \mathbb{E} \left ( \sum_{i=1}^n g_i\right )^{2d} = n^d \mathbb{E} g_1^{2d}=n^d (2d-1)!!, \end{align*}
	where $(d-1)!!:=(d-1) (d-3) \cdots 3 \cdot 1$.
	Next, given support sets $S$ and $\tilde S$, observe that \begin{align*}
	\langle \theta, \tilde \theta \rangle | S, \tilde S  \overset{d}{=} \frac{\Delta^2}{s} \sum_{i \in S \cap \tilde S} R_i,  \notag 
	\end{align*}
	where $R_i$ are i.i.d. Rademacher random variables again. Hence, arguing as before using comparison to Gaussian random variables, we obtain
	\begin{align*}
	\mathbb{E} \langle \theta, \tilde \theta \rangle^{2d} & = \left (\frac{\Delta^2}{s} \right )^{2d}\mathbb{E} \left [ \mathbb{E}\left ( \sum_{i \in S \cap \tilde S} R_i \right )^{2d} \bigg | S, \tilde S \right ] \\
	& \leq (2d-1)!! \left (\frac{\Delta^2}{s} \right )^{2d} \mathbb{E} |S \cap \tilde S |^d. 
	\end{align*}
	Define $Z_{i}=\mathbf{1}( i \in \mathcal{S}), \tilde Z_i=\mathbf{1}(i\in \mathcal{\tilde S})$ and observe that
	\begin{align*}
	|S \cap \tilde S | = \sum_{i=1}^p Z_i \tilde Z_i.
	\end{align*}
We have that $\mathbb{E} Z_i=\mathbb{E}\tilde Z_i=\frac{s}{p}$ and hence the $Z_i$ are Bernoulli random variables with success probability $s/p$,  but not independent as they are drawn from a finite population. Conditionally on $\tilde S$, we have that $|S \cap \tilde S|=\sum_{i \in \tilde S} Z_i$ is a sum of $|\tilde S|$ random variables drawn at random without replacement.  Moreover, $|\cdot |^d$ is a continuous convex function. Hence, by the tower property of expectation and by applying Theorem 4 in \cite{Hoeffding63} (which compares the moments of random samples drawn with and without replacement) we obtain that 
%
\begin{align*}
\mathbb{E} \left | \sum_{i=1}^p Z_i \tilde Z_i\right |^d & =  \mathbb{E} \left [ \mathbb{E} \left |\sum_{i \in \tilde S} Z_i \right |^d\bigg | \tilde S \right ] \\ &  \leq \mathbb{E} \left [ \mathbb{E} \left |\sum_{i \in \tilde S} B_i \right |^d\bigg | \tilde S \right ] 
= \mathbb{E} \left|\sum_{i=1}^s B_i\right|^d,
\end{align*}
where $B_i$ are i.i.d. Bernoulli random variables with success probability $s/p$ each. 
Since the $B_i$ are independent, bounded by one and have variance bounded by $s/p$ we obtain by Bernstein's inequality that

\begin{align*}
\mathbb{P} \left ( \left |  \sum_{i=1}^s (B_i - \mathbb{E} B_i) \right | > t  \right ) \leq 2\exp \left ( - \frac{t^2}{2\frac{s^2}{p}+\frac{2}{3}t}\right ). 
\end{align*}
Hence, by the triangle inequality for the $\ell_d$-norm, we obtain that 
\begin{align*}
\left [\mathbb{E} \left | \sum_{i=1}^s B_i\right |^d \right ]^{1/d} & \leq \frac{s^2}{p} + \left [\mathbb{E} \left | \sum_{i=1}^s (B_i - \mathbb{E} B_i) \right |^d \right ]^{1/d} \\
& = \frac{s^2}{p} + \left [\int_0^{\infty} \mathbb{P} \left (  \left | \sum_{i=1}^s (B_i - \mathbb{E} B_i)\right | > t^{1/d}\right ) \text{d}t \right ]^{1/d} \\
& \leq \frac{s^2}{p} + 2^{1/d}\left [\int_0^{\infty} \exp \left ( -\frac{t^{2/d}}{2\frac{s^2}{p}+\frac{2}{3}t^{1/d}}\right ) \text{d}t \right ]^{1/d}. 
\end{align*}
We bound the integral above. Indeed, we have that
\begin{align*}
\int_0^{\infty} \exp \left ( -\frac{t^{2/d}}{2\frac{s^2}{p}+\frac{2}{3}t^{1/d}}\right ) \text{d}t & \leq \int_0^{\infty} \exp \left ( -\frac{t^{2/d}}{4\frac{s^2}{p}}\right ) \text{d}t + \int_0^{\infty} \exp \left ( -\frac{3t^{1/d}}{4}\right ) \text{d}t \\
& \leq \sqrt{\frac{\pi}{2}}\left(\frac{2ds^2}{p}\right)^{d/2} + \sqrt{\frac{\pi}{2}}\left (\frac{4d}{3} \right )^d,
\end{align*}
where we have used the following identity (which can be obtained from the definition of the Gamma function): for all $a > 0$ and $b \ge 1/2$,
\[ \int_0^\infty \exp(-ax^{1/b}) \text{d}x = \frac{b\Gamma(b)}{a^b} = \frac{\Gamma(b+1)}{a^b} \le \sqrt{\frac{\pi}{2}}\left(\frac{b}{a}\right)^b. \]
	Hence, overall we obtain that
	\begin{align*}
	\|L_n^{\leq D} \|^2& -1 =\Ex_{(\theta, z), (\tilde \theta, \tilde z) \sim \pi} \sum_{d=1}^{\lfloor D/2 \rfloor} \frac{1}{(2d)!} \langle z, \tilde z \rangle^{2d} \langle \theta, \tilde \theta \rangle^{2d} \\
	& \leq \sum_{d=1}^{\lfloor D/2 \rfloor} \left(\frac{n\Delta^4}{s^2}\right)^d \left\{\frac{s^2}{p} + 2^{1/d} \left(\frac{\pi}{2}\right)^{\frac{1}{2d}} \left[\left(\frac{2ds^2}{p}\right)^{d/2} + \left(\frac{4d}{3}\right)^d\right]^{1/d}\right\}^d \\
	& \leq \sum_{d=1}^{\lfloor D/2 \rfloor} \left(\frac{n\Delta^4}{s^2}\right)^d \left\{\frac{s^2}{p} + 2^{1/d} \left(\frac{\pi}{2}\right)^{\frac{1}{2d}} \left(\sqrt{\frac{2ds^2}{p}}  + \frac{4d}{3}\right)\right\}^d \\
	& \leq \sum_{d=1}^{\lfloor D/2 \rfloor}  \left ( \frac{n \Delta^4}{p} +\frac{4n \Delta^4 D}{s^2}+\frac{4 n \Delta^4 D^{1/2}}{p^{1/2}s}\right )^d \\
	&= \sum_{d=1}^{\lfloor D/2 \rfloor}  \left (  \sqrt{\frac{n \Delta^4 }{p}}+ \sqrt{\frac{4n \Delta^4 D}{s^2}} \right )^{2d} =O \left ( {\frac{n \Delta^4 }{p}}+ {\frac{4n \Delta^4 D}{s^2}}\right )
	\end{align*}
	per our assumptions and where we used that $((2d-1)!!)^2/(2d)! \leq 1$.
\end{proof}
\subsection{Proof of Theorem \ref{thm:complowerz}}

\begin{proof}
Assuming that \eqref{thm:complowerz power} holds and given data $[X_1, \dots, X_n]=:X$ either generated from $\mathbb{P}_0$ or $\mathbb{P}_{\theta,z}$, $(\theta (1-\epsilon), z) \in \mathcal{P}_\Delta$, we perform the following sample splitting trick: we denote $E=[\varepsilon_1,\dots, \varepsilon_n]$ and generate $\tilde E$ such that $\tilde E \overset{d}{=}E$ and $\tilde E$ and $E$ are independent and for $\epsilon \in (0,1)$ define
	\begin{equation*}
	X^{(1)}= \frac{X+\frac{1}{\epsilon}\tilde E}{\sqrt{1+\frac{1}{\epsilon^2}}}~~~~~\tilde X^{(2)}=\frac{X-\epsilon \tilde E }{\sqrt{1+\epsilon^2}}.
	\end{equation*}
	Since for fixed $(\theta, z)$ \begin{align} \label{bound cov ss} \text{Cov} \left (	X^{(1)}, 	X^{(2)} \right )=\frac{1}{\sqrt{2+\epsilon^2+\frac{1}{\epsilon^2}}}\text{Cov} \left (E+\frac{1}{\epsilon}\tilde E, E-\epsilon \tilde E \right )=0,
	\end{align} and $E$ and $\tilde E$ are Gaussian, we obtain that $	X^{(1)}$ and $	X^{(2)}$ are independent. Moreover, since $	X^{(2)}=(\theta/\sqrt{1+\epsilon^2})z^T +(E- \epsilon \tilde E)/\sqrt{1+\epsilon^2}$, $(\theta/\sqrt{1+\epsilon^2}, z) \in \mathcal{P}_\Delta$ (as $1-\epsilon^2 \leq  1/\sqrt{1+\epsilon^2})$ and $(E-\epsilon \tilde E)/\sqrt{1+\epsilon^2}\overset{d}{=}E$ we see that $	X^{(2)}$ can be viewed as generated from a parameter in $\mathcal{P}_\Delta$. Hence, by assumption in \eqref{thm:complowerz power}, we have that $$\mathbb{P}_{\theta,z} \left ( \ell(\hat z(	X^{(2)}), z) > \delta \right )\leq 1-\alpha $$ and by construction $\hat z=\hat z(	X^{(2)})$ is independent of $	X^{(1)}$ and therefore $\hat z$ and $(E+\frac{1}{\epsilon} \tilde E)$ are also independent. Next define a test statistic
	\begin{equation*}
	T_n^2:= \sum_{i=1}^s \left [ \left ( \frac{ X^{(1)}\hat z}{n}\right )_{(i)} \right ]^2,
	\end{equation*}
	where $a_{(i)}$ denotes the $i$-th largest (in absolute value) element of $a$. Note that $(E+\frac{1}{\epsilon}\tilde E)\hat z/\sqrt{1+\frac{1}{\epsilon^2}}  \thicksim \mathcal{N}(0, n I_p)$. Hence, by Proposition E.1. in \cite{BellecLecueTsybakov18}, we have that with probability at least $1-s/(2p)$
	\begin{align*}
	    \sum_{i=1}^s \left [\left (\frac{ \left (E+\frac{1}{\epsilon}\tilde E \right )\hat z}{n\sqrt{1+\frac{1}{\epsilon^2}}}\right )_{(i)} \right ]^2  \leq \frac{6 s \log(ep/s)}{n}.
	\end{align*}

	Thus, 
	under $H_0: \theta=0$ we bound with probability at least $1-s/(2p)$
	\begin{equation*}
	T_n^2 = \sum_{i=1}^s \left [\left (\frac{ \left (E+\frac{1}{\epsilon}\tilde E \right )\hat z}{n\sqrt{1+\frac{1}{\epsilon^2}}}\right )_{(i)} \right ]^2 \leq \frac{6s \log(ep/s)}{n}.
	\end{equation*}
	Now consider the alternative $H_1: (\theta/\sqrt{1+\epsilon^2}, z) \in \mathcal{P}_\Delta$. Denote by $i(\theta)$ the index that corresponds to the $i$-th largest element of $\theta$. 
	Hence, on the event $$\{ \hat \ell(\hat z(X^{(2)}, z)) \leq \delta \} \cap  \left \{ 	    \sum_{i=1}^s \left [\left (\frac{ \left (E+\frac{1}{\epsilon}\tilde E \right )\hat z}{n\sqrt{1+\frac{1}{\epsilon^2}}}\right )_{(i)} \right ]^2  \leq \frac{6 s \log(ep/s)}{n} \right \}$$ we obtain that
	\begin{align*}
	T_n & \geq \left ( \sum_{i=1}^s \left [\left ( \frac{X^{(1)} \hat z}{n}\right )_{i(\theta)} \right ]^2 \right )^{1/2} \\ &  \geq \left \| \frac{\theta z^T \hat z}{n\sqrt{1+\frac{1}{\epsilon^2}}} \right \|-\left ( \sum_{i=1}^s \left [ \left (\frac{ (E+\frac{1}{\epsilon}\tilde E)\hat z}{n\sqrt{1+\frac{1}{\epsilon^2}}} \right )_{i(\theta)} \right  ]^2 \right )^{1/2} \\ 
	& \geq \frac{\| \theta \|}{\sqrt{1+\frac{1}{\epsilon^2}}} (1-2\ell(\hat z, z) ) - \sqrt{\frac{6s\log(ep/s)}{n}}   \\ & \geq \frac{\Delta}{\sqrt{1+\frac{1}{\epsilon^2}}}  (1-2\delta) - \sqrt{\frac{6s\log(ep/s)}{n}}.
	\end{align*}
	Hence, for $$\Delta > 2 \sqrt{6}(1-2\delta)^{-1} \sqrt{1+\frac{1}{\epsilon^2}} \sqrt{\frac{s \log(ep/s)}{n}},$$ defining the test
	$${t_n(X)}:=\begin{cases} 1 ~~~~~~~\text{if}~~~T_n^2 > \frac{6s \log(ep/s)}{n} \\
	0 ~~~~~~~\text{else} \end{cases}
	$$
	we have that 
	\begin{align*}
 \mathbb{E}_0 [t_n(X)] + \sup_{(\theta/\sqrt{1+\epsilon^2}, z) \in \mathcal{P}_\Delta} \mathbb{E}_{\theta,z} \left [1-t_n(X)  \right ] & \leq (p/s)^{-1}+1-\alpha.
	\end{align*}
Finally, after having obtained  $\hat z$ which has complexity  $O(r_n)$ by assumption,  the complexity of calculating $t_n$ is dominated by the matrix-vector multiplication $X^{(1)} \hat z$ which has complexity $O(np)$.
\end{proof}

\appendix

\section{Proof of Lemma \ref{Lemma rates selec proj}}
\begin{proof}
Note that 
$$ \mathbb{E} \left [\frac{XX^T}{n}-I_p \right ] = \theta \theta^T=\|\theta\|^2P.$$
We first control the deviations from the mean in $\ell_\infty$-norm. Indeed, denoting $X=[X_1, \dots, X_n]$ and $E=[\varepsilon_1, \dots, \varepsilon_n]$ and $z=(z_1, \dots, z_n)^T$, we have that
\begin{align*}
\left \| \frac{XX^T}{n}-I_p-\theta \theta^T \right \|_\infty \leq \left \| \frac{EE^T}{n}-I \right \|_\infty + 2 \left \| \frac{\theta z^TE^T}{n} \right \|_\infty=:T_1+T_2
\end{align*}
	The first term $T_1$ can be bound using Bernstein's inequality and a union bound with probability $1-p^{-2}$ by $c \sqrt{\log(p)/n}$, compare e.g. the proof of Lemma D.1 in \cite{VuLei13}, noting that the assumption $\log(p)\leq n$ ensures that the Gaussian component in the exponential inequality dominates.
	For the second term $T_2$ we use Gaussian concentration directly. Indeed, we have that
	\begin{align*}
	(\theta z^T E^T)_{jl} = \sum_{i=1}^n z_i \theta_j \varepsilon_{il} \thicksim \mathcal{N}(0,n \theta_j^2),
	\end{align*}
	and hence, using a union bound and Gaussian concentration we obtain that with probability at least $1-p^{-2}$
	\begin{equation*}
	T_2\leq c' \kappa  \sqrt{\log(p)/n } . 
	\end{equation*}
	Denote $M=\theta \theta^T$ and $\hat M=XX^T/n-I_p$. 
Hence, overall, with probability at least $1-2p^{-2}$ the following event occurs
\begin{align*}
\Omega:= \left \{ \left \| \frac{XX^T}{n}-I_p-\theta \theta^T \right \|_\infty < C (1+\kappa)\sqrt{\log(p)/n} \right \} = \{ \| M-\hat M \|_\infty < \lambda \}.
\end{align*}
For the rest of the proof suppose that we work on $\Omega$.
Since $P \in \mathcal{F}^1$ we have by definition of the objective function that
\begin{equation*}
\langle \hat M, \hat P \rangle - \lambda \|\hat P\|_1 \geq \langle \hat M, P \rangle - \lambda \|P \|_1
\end{equation*}
Using curvature Lemma 3.1 from \cite{VuLei13} for the first inequality  and afterwards the above inequality we obtain that 
\begin{align*}
\| \hat P - P \|_F^2 & \leq \frac{2}{\|\theta\|^2} \langle M, P-\hat P \rangle = \frac{2}{\|\theta\|^2} \left [ \langle \hat M, P-\hat P \rangle + \langle M-\hat M, P-\hat P \rangle \right ] \\
& \leq \frac{2}{\|\theta\|^2} \left [\lambda \|P\|_1-\lambda \|\hat P\|_1 + \|M-\hat M \|_\infty \|P-\hat P\|_1  \right ] \\
& \leq \frac{2 \lambda }{\|\theta\|^2} \left [ \|P_S\|_1 - \|\hat P_S\|_1 + \|(P-\hat P)_S\|_1 \right ] \\
& \leq \frac{4 \lambda s \|P-\hat P\|_F }{\|\theta\|^2}.
\end{align*}
This shows the first assertion. We next prove false positive control. Recall that
\begin{align} \label{objective}
\hat P=\argmax_{P \in \mathcal{F}^1} \langle \hat M, P \rangle - \lambda \|P\|_1. 
\end{align}
We first show that there exists at least one $\tilde P \in \mathcal{F}^1$ such that 
\begin{align*}
    \langle \hat M, \tilde P \rangle - \lambda \|\tilde P\|_1 \geq 0.
\end{align*}
Indeed, setting $\tilde P=P=\theta \theta^T /\|\theta\|^2$, we have that 
\begin{align*}
        \langle \hat M, \tilde P \rangle - \lambda \|\tilde P\|_1 \geq \|\theta\|^2 - \|\hat M-M\|_\infty \|\tilde P\|_1 - \lambda \|\tilde P\|_1 \geq \|\theta\|^2 - 2\lambda s \geq 0,
\end{align*}
where we used that we work on $\Omega$ and that $\|\theta\|^2 \geq 2 \lambda s$. 
Hence, it suffices to consider $\tilde P \in \mathcal{F}^1$ such that  $ \langle \hat M,  \tilde P \rangle - \lambda \| \tilde P\|_1 \geq 0$ and show for those that have support with indices outside of $S$ leads to a strictly suboptimal objective value. 

We first show that any $\tilde P$ which does not select at least one coordinate in $S$ has negative objective value. Indeed, we have for $\tilde S $ with $\text{supp}(\tilde S) \subset S^c$ that 
\begin{align*}
    \langle \hat M, \tilde P \rangle - \lambda \|\tilde P\|_1 = \langle \hat M-M, \tilde P \rangle - \lambda \|\tilde P\|_1 \leq \|\hat M-M\|_\infty \|\tilde P\|_1 - \lambda \|\tilde P\|_1 < 0,
\end{align*}
where we used again that on $\Omega$ we have $\|\hat M-M\|_\infty < \lambda$. 
 Hence, if $  \langle \hat M,  \tilde P \rangle - \lambda \| \tilde P\|_1 \geq 0$ then $\tilde P_{jj}$ is positive for at least one coordinate in $S$.  Observe that
\begin{align*}
\langle \hat M, \tilde P \rangle - \lambda \|\tilde P\|_1 = \langle \hat M_{S}, \tilde P_{S} \rangle -\lambda \|\tilde P_S\|_1 + \sum_{(i,j) \notin S \times S} \hat M_{ji} \tilde P_{ij}-\lambda |\tilde P_{ji}|.
\end{align*} 
Recalling that we work on the event $\Omega$ where $\|\hat M-M\|_\infty < \lambda$, we have for $\tilde P$ with $\text{supp}(\tilde P) \not \subset S$ that
\begin{align*}
\sum_{(i,j) \notin S \times S} \hat M_{ji}\tilde P_{ij}-\lambda |\tilde P_{ji}| & \leq \sum_{(i,j) \notin S \times S} (|\hat M_{ij}|-\lambda) |\tilde P_{ij}| \\
& \leq  \sum_{(i,j) \notin S \times S} (\|\tilde M-M\|_\infty -\lambda) |\tilde P_{ij}|  < 0. 
\end{align*}
Since $   \langle \hat M,  \tilde P \rangle - \lambda \| \tilde P\|_1 \geq 0$ there exists at least one $j \in S$ such that $\tilde P_{jj} > 0$ and hence
 we have $\tr (\tilde P_S)>0$. 
Hence, we can define $\check P:=\tilde P_S/\tr (\tilde P_S)$ and have that
\begin{align*}
    \langle \hat M , \tilde P \rangle - \lambda \|\tilde P\|_1 & <  \langle \hat M_S, \tilde P_S \rangle - \lambda \|\tilde P_S\|_1 \\ & \leq 
        \frac{1}{\tr (\tilde P_S)} \left [ \langle \hat M, \tilde P_S, \rangle - \lambda \|\tilde P_S\|_1 \right ] = \langle \hat M, \check P \rangle - \lambda \|\check P\|_1.
\end{align*}
It is left to check that $\check P \in \mathcal{F}^1$. Indeed, we have that $\check P^T=\check P$ as $\tilde P$ is symmetric, $\tr(\check P) =1$, $\check P$ is positive semi-definite as $x^T \check Px \geq  x_S^T \tilde P x_S \geq 0$, where we used that $\tilde P$ is positive semi-definite and $\|\check P\| \leq \tr (\check P) = 1$.  Hence, $\check P \in \mathcal{F}^1$. We conclude that for any $\tilde{P} \in \mathcal{F}^1$ with positive objective value and support not entirely in $S$ we can find some $\check P \in \mathcal{F}^1$ with support in $S$ that has strictly larger objective value. Since there exists at least one feasible solution with positive objective value, we conclude that any maximizer of the objective must have support in $S$ and hence $\text{supp}(\hat P )\subset S.$\end{proof}

\section{Proof of Theorem \ref{theorem upper 2}}
\begin{proof}
We use Algorithm \ref{alg 2} to construct $\tilde z$ and $\hat z$. Note that Algorithm \ref{alg 2} has complexity bounded by $O(np)$ and hence is a polynomial-time algorithm. This is due to the fact that $\hat \theta$ can be computed explicitly by setting the $p-s$ smallest entries of $X^{(2)}\tilde z/n$ to zero. 

 \begin{algorithm}[h] \label{alg 2}
	\SetAlgoLined
 	\KwIn{Data matrix $X=[X_1, \dots, X_n] \in\mathr^{p\times n}$, sparsity level $s$, selection parameter $\kappa$}
 	\KwOut{Clustering label vector $ \hat  z\in \{-1,1\}^n$}
 	\nl Sample $$\tilde E_{ij} \overset{i.i.d.}{\thicksim} \mathcal{N}(0,1), ~~~~~~~~\check{E}_{ij}\overset{i.i.d.}{\thicksim} \mathcal{N}(0,2).$$ \\
 	\nl Compute $$ X^{(1)}=X-\tilde E-\check E, ~~~ X^{(2)}=X-\tilde E+\check E, ~~~~X^{(3)}=X+\tilde E$$. \\
 	\nl Estimate largest entry of $\theta$ by diagonal thresholding
 $$\hat k:= \argmax_{k \in [p]} (X^{(1)} (X^{(1)})^T)_{kk}. $$ \\ 
 	\nl Compute preliminary cluster estimator 
 	$$ \tilde z_i:= \text{sgn}( X^{(1)}_{i\hat k}), ~i=1, \dots, n$$ \\
 \nl Use second sample to estimate $\theta$ by considering only $s$ largest entries 
 $$
 \hat \theta=\argmin_{\theta:~\|\theta\|_0\leq s} \left  \| \frac{X^{(2)}\tilde z}{n} - \theta  \right \|.$$\\ \nl Use third sample to improve clustering and return for $i=1, \dots, n$
 	\begin{equation} \hat  z_i=\text{sgn}(\langle \hat \theta, X^{(3)}_i\rangle). \end{equation}
 	\caption{Sparse clustering with sample splitting}\label{alg:2}
\end{algorithm}

We first show that $X^{(1)}, X^{(2)}$ and $X^{(3)}$ are independent. Indeed, $X^{(1)}, X^{(2)}$ and $X^{(3)}$  are all Gaussian and hence independence is implied by zero covariance structure. We have that
\begin{align*}
    \text{Cov}(X^{(3)}, X^{(2)})=\text{Cov}(E+\tilde E, E-\tilde E+\check E)=0,
\end{align*}
and likewise for the other pairwise comparisons. 

Next, we show that $\hat k$ selects with high probability a coordinate such that $|\theta_{\hat k}| \geq \kappa/2$. We observe that
\begin{align*}
    \hat k=\argmax_{k\in [p]} \left [ \frac{(X^{(1)}(X^{(1)})^T)_{kk}}{4n}-1 \right ]. 
\end{align*}
We have that
\begin{align*}
\frac{(X^{(1)} (X^{(1)})^T)_{kk}}{{4}n}-1\overset{d}{=} \frac{\theta_k^2}{4} +  \frac{\theta_k (z^TE^T)_k}{{2}n}+ \left (\frac{EE^T}{n}-I \right )_{kk}. 
\end{align*}
Hence, using Bernstein's inequality to bound the last term above and a Gaussian tail bound to bound the term in the middle above, we obtain that for some constant $c>0$ with probability at least $1-p^{-2}$
\begin{align}
 \frac{\theta_k^2}{4} - c (1+|\theta_k|)\sqrt{\frac{\log(p)}{n}}   \leq  \frac{(X^{(1)} (X^{(1)})^T)_{kk}}{{4}n}-1 \leq \frac{\theta_k^2}{4}+ c (1+|\theta_k|)\sqrt{\frac{\log(p)}{n}} \label{bound thetaj}.
\end{align}
By a union bound \eqref{bound thetaj} occurs for all $k \in [p]$ simultaneously with probability at least $1-p^{-1}$. In particular, for $k \in [p]$ such that $|\theta_k| < \kappa/2$, we have on the corresponding event that
\begin{align*}
   \frac{(X^{(1)} (X^{(1)})^T)_{kk}}{4n}-1 \leq \frac{\kappa^2}{16}+ c \left (1+\frac{\kappa}{2} \right )\sqrt{\frac{\log(p)}{n}} < \kappa^2/8.
\end{align*}
Moreover, by assumption there exists at least one index $k$ such that $|\theta_k| \geq \kappa$ such that we have on the above event
\begin{align*}
   \frac{(X^{(1)}(X^{(1)})^T)_{kk}}{{4}n}-1  \geq \frac{\theta_k^2}{4} - c(1+|\theta_k| \sqrt{\frac{\log(p)}{n}}\geq \theta_k^2/8 \geq \kappa^2/8.
\end{align*}
Here we used in both cases that $\log(p)\leq n$ and that $\kappa$ is large enough. 
We conclude that $\hat k$ selects an index $k \in [p]$ such that with probability at least $1-p^{-1}$ we have that $|\theta_{\hat k}|\geq \kappa/2$.

Next, we show that on this event the preliminary cluster estimator $\tilde z$ clusters better than with a random guess. Denote $\varepsilon^{(1)}_i:=E_i-\tilde E_i-\check E_i \thicksim \mathcal{N}(0,4I_p)$. 
Setting $\pi=\text{sgn}(\theta_{\hat k}$ and using Markov's inequality we obtain that with probability at least $1-p^{-1}$
\begin{align}
 \notag    \ell(\tilde z, z ) &  = \min_{\pi \in \{-1,1\}} \frac{1}{n} \sum_{i=1}^n \mathbf{1} \left ( \pi \text{sgn} (z_i \theta_{ \hat k}+\varepsilon^{(1)}_{i \hat k}) \neq z_i) \right ) \leq \frac{1}{n} \sum_{i=1}^n \mathbf{1} \left ( | \theta_{ \hat k}| \leq | \varepsilon^{(1)}_{i \hat k}|  \right ) \\  & \leq \frac{1}{n} \sum_{i=1}^n \frac{(\varepsilon^{(1)}_{i \hat k})^2}{\theta_{ \hat k}^2} \leq \frac{4}{\kappa^2 n} \max_{k \in [p]} \sum_{i=1}^n \frac{(\epsilon_{ik}^{(1)})^2}{4}. \notag
\end{align}
For fixed $k \in [p]$  we have that $\sum_{i=1}^n {(\varepsilon_{ik}^{(1)})^2}/4$ is chi-square distributed with $n$ degrees of freedom. Hence, by Bernstein's inequality and then using a union bound we obtain that with probability at least $1-p^{-1}$
$$\max_{k \in [p]} \frac{1}{n}\sum_{i=1}^n {(\varepsilon_{ik}^{(1)})^2}-4 \lesssim \sqrt{\frac{\log(p)}{n}}. $$
Hence, using another union bound, we have with probability at least $1-2p^{-1}$
\begin{align}
    \label{bound tilde z}
    \ell(\tilde z, z ) \lesssim \frac{1}{\kappa^2} \left (1+ \sqrt{\frac{\log(p)}{n}} \right ) \lesssim \frac{1}{\kappa^2}. 
\end{align}

Next, we consider the second sample $X^{(2)}$, conditionally on $\tilde z$, and prove bounds for the estimation error of $\hat \theta$. We have that
\begin{align} \label{eq norm means}
\frac{X^{(2)} \tilde z}{n}= \theta \frac{\langle z, \tilde z \rangle}{n} + \frac{(E-\tilde E + \check E)\tilde z}{n}.
\end{align}
Since $\tilde z$ and $X^{(2)}$ are independent, $E-\tilde E + \check E$ and $\tilde z$ are independent, too, and we recognize \eqref{eq norm means} conditionally on $\tilde z$ as an instance of the $p$-dimensional normal means model with coordinatewise variance  $4/n$. 
Hence, by Corollary 2.8 in \cite{RigolletHuetter17} (with $\text{MSE}(\cdot)=\| \theta \langle \tilde z, z \rangle/n- \cdot \|^2/p$, $\mathbb{X}=I_p$, $k=s$, $\hat \theta^{LS}_{\mathcal{B}_0(k)}=\hat \theta$, $\sigma^2=4/n$ , $n=p$, $d=p$ and $\delta=s/p$ there) we have with probability at least $1-(s/p)$ that 
\begin{align*}
 \left    \| \theta \frac{\langle \tilde z, z \rangle}{n} -\hat \theta \right \|^2 \lesssim  \frac{s \log(ep/s)}{n}. 
\end{align*}
Using a union bound to incorporate the event where $\ell(\hat z, z) \lesssim \kappa^2$, we conclude that with probability at least $1-(3s)/p$
\begin{align}
   \min_{\pi \in \{-1,1\}} \|\hat \theta-\pi \theta\| &  \lesssim \sqrt{\frac{s \log(ep/s}{n}}+\|\theta\| \ell(\tilde  z,z) \notag \\ & \lesssim \sqrt{\frac{s \log(ep/s)}{n}} + \frac{\|\theta\|}{\kappa^2}. \ \label{bound hat theta 2}
\end{align}
We finally bound the misclustering error of $\hat z$. Note that since $\hat \theta$ only depends on $ X^{(1)}$ and $X^{(2)}$ that $\hat \theta$ and $ X^{(3)}$ are independent.
Similar as in the proof of Theorem \ref{thm mainsimple} we work on the event $\tilde \Omega:= \{ \min_{\pi \in \{-1,1\}} \|\hat \theta-\pi \theta\| \leq c (\sqrt{s \log(ep/s)/n}+\|\theta\|/\kappa^2) \}$, which by \eqref{bound hat theta 2} occurs with probability at least $1-(3s)/p$, consider each $z_i$ separately and use Markov's inequality to conclude. We assume again without loss of generality that $z_i = 1$ and pick $\pi$ as the minimizing $\pi$ in \eqref{bound hat theta 2} for all $i$, i.e. $\pi=\text{sgn}(\langle \hat \theta, \theta \rangle)$. Denote $ \varepsilon_i^{(3)} := E_i+\tilde E_i \thicksim \mathcal{N}(0, 2I_p)$. Hence, by independence, a Gaussian tailbound and the tower property of expectation we obtain that for some constant $c'>0$
\begin{align}
   \mathbb{E} \left [\mathbf{1} \left ( z_i \neq \pi \hat z_i \right ) \mathbf{1}_{\tilde \Omega} \right ]& = \mathbb{E} \left [ \mathbf{1}( \langle \pi \hat \theta, \theta +  \varepsilon_i^{(3)} \rangle \leq 0 ) \mathbf{1}_{\tilde \Omega} \right ] \leq \mathbb{E} \left [\exp \left (- \frac{\langle \pi \hat \theta, \theta \rangle^2}{4 \|\hat \theta\|^2}\right ) \mathbf{1}_{\tilde \Omega} \right ]\notag \\
   &  \leq \mathbb{E} \left [\exp \left (  - \frac{ \|\theta\|^2 ( \|\theta\|^2-\|\pi \hat \theta-\theta\|^2)}{4(\|\theta\|+\|\pi \hat \theta-\theta\|)^2} \right ) \mathbf{1}_{\tilde \Omega} \right ] \notag \\ & \leq \mathbb{E} \left [\exp \left (-\frac{\|\theta\|^2}{4} \left (1-2\frac{\|\pi \hat \theta-\theta\|}{\|\theta\|} \right ) \right )\mathbf{1}_{\tilde \Omega}  \right ] \notag \\
   & \leq \exp \left ( - \frac{\|\theta\|^2}{4} \left (1- c' \left ( \sqrt{\frac{s \log(ep/s)}{n\|\theta\|^2}} + \frac{1}{\kappa^2} \right ) \right ) \right ) . \notag 
\end{align}
Finally, applying Markov's inequality and a union bound as in the proof of Theorem \ref{thm mainsimple}, we obtain the desired result.
\end{proof}

\section*{Acknowledgements.} The authors would like to thank two anonymous referees for their helpful comments and remarks, which lead to an improved version of the manuscript. Moreover, M. L\"offler would like to thank Anderson Ye Zhang and Sara van de Geer for helpful discussions and Martin Wahl for careful proofreading of the manuscript.

{
\bibliographystyle{alpha}
\bibliography{spectral}}

\end{document}